\documentclass{higher-birkhoff-2}
\usepackage{calligra}
\usepackage{subcaption}
 \usepackage[mathscr]{eucal}
 \usepackage[scr=esstix,cal=boondox]{mathalfa}
 \usepackage{hyperref}

\title{Higher dimensional Birkhoff attractors}
\author{Marie-Claude Arnaud}\thanks{MCA: Universit\'e Paris Cit\'e and Sorbonne Universit\'e, CNRS, IMJ-PRG, F-75005 Paris, France. Supported by ANR CoSyDy, ANR21-CE40-0014. Email:  marie-claude.arnaud@imj-prg.fr}
\author{Vincent Humili\`ere} \thanks{VH: Sorbonne Universit\'e and Universit\'e Paris Cit\'e, CNRS, IMJ-PRG, F-75005 Paris, France. Member of IUF and supported by COSY (ANR-21-CE40-0002) and CoSyDy (ANR-21-CE40-0014). Email: vincent.humiliere@imj-prg.fr}
\author{Claude Viterbo}
\thanks{CV: Universit\'e Paris-Saclay, CNRS, Laboratoire de math\'ematiques d'Orsay, 91405, Orsay, France. Supported by ANR COSY (ANR-21-CE40-0002). Email: claude.viterbo@universite-paris-saclay.fr}
\thanks{MZ: Sorbonne Universit\'e and Universit\'e Paris Cit\'e, CNRS, IMJ-PRG, F-75005 Paris, France. Supported by ANR CoSyDy, ANR-CE40-0014. Email: maxime.zavidovique@imj-prg.fr}

\begin{document}

\maketitle

\begin{center}
  With an appendix by \textsc{Maxime Zavidovique}
\end{center}

\begin{abstract}
  We extend to higher dimensions the notion of Birkhoff attractor of a dissipative map. We prove that this notion coincides with the classical Birkhoff attractor defined by Birkhoff in \cite{Birkhoff-attractor}. 
  We prove that for the dissipative system associated to the discounted   Hamilton-Jacobi equation
the graph of a solution is contained in the Birkhoff attractor. 
We also study what happens when we perturb a Hamiltonian system to make it dissipative and let the perturbation go to zero. 
The paper contains two main results on $\gamma$-supports and elements of the $\gamma$-completion of the space of exact Lagrangians. Firstly the $\gamma$-support of a Lagrangian in a cotangent bundle carries the cohomology of the base 
and secondly given an exact Lagrangian, $L$, any  Floer theoretic equivalent Lagrangian $L'$ is the $\gamma$-limit of Hamiltonian images of $L$. The appendix provides instructive counter-examples.
\end{abstract}

\tableofcontents

\section{Introduction and main results}

This work is dedicated to the study of conformal symplectic dynamics, a now classical
extension of symplectic dynamics where the symplectic form is only preserved, up to a multiplicative factor. 
Such dynamics include for instance 
mechanical systems with a friction force proportional
to velocity. 

 Dissipative twist maps of the two-dimensional annulus form another class of examples of conformally symplectic dynamical systems which was extensively studied by Birkhoff and others \cite{Birkhoff-attractor,Charpentier-1,Charpentier-2,LeCalvez-Birkhoff1, LeCalvez-Birkhoff2}. In this setting, there is a global attractor but, more interestingly, the  core of the dynamics is carried by a subset of the global attractor, the so-called \emph{Birkhoff attractor}. This is a minimal invariant compact subset which separates the annulus. Note that despite its name it is not in general an attractor, since some points may not converge to it. There are very few results in the higher dimensional case. Invariant submanifolds were studied in \cite{Arnaud-Fejoz}, but it may happen that there is no such invariant submanifold. One difficulty is that there may be no global attractor and if it exists, it may not separate. It is thus a priori unclear how to define an invariant subset that would extend the construction of Birkhoff.

 The main goal of this article is to provide such an extension of the Birkhoff attractor in any dimension, even when the dynamics has no attractor. The construction relies on tools from symplectic topology, and more precisely on the completion of the space on exact Lagrangian submanifolds with respect to a certain metric. The Birkhoff attractor arises as the ``support'' (called $\gamma$-support) in a sense to be made precise later of a generalized Lagrangian in this completion that is a fixed point of the dynamics.

 We then study this generalized attractor from various viewpoints.
 From the analytic viewpoint, we compare this attractor with the graph of the so-called viscosity solution of the discounted Hamilton-Jacobi equation in the setting of weak KAM theory. Driven by the beautiful result from \cite{DFIZ} on the selection of weak KAM solutions, there was a hope that  the generalized Lagrangian fixed point could converge as conformal exponent goes to zero, thus selecting a generalized invariant Lagrangian of the symplectic (non conformal) dynamics. We prove with some example that this is unfortunately false.
 
Let us observe that the works \cite{Birkhoff-attractor,Charpentier-1,Charpentier-2} also provide a link between the restricted dynamics on the Birkhoff attractor (more precisely its rotation set) and the topological complexity of the attractor. In higher dimension, there seems to be no hope to obtain such a result. Indeed, it is known that for any flow $(\psi_t)$ on a manifold $M$ there exists a conformal flow on the cotangent bundle $T^*M$ for which the zero section is the Birkhoff attractor and whose restriction to the zero section is $(\psi_t)$; see \cite[Sect. 7.1]{Arnaud-Fejoz}. 

A second part of the paper deals with symplectic topological aspects of the $\gamma$-supports and so in particular the Birkhoff attractor. A first result is that a compact $\gamma$-support  and in particular a Birkhoff attractor carries the cohomology of the base, which can be seen as the higher dimensional analogue of ``separating''.
A second result is to prove a weak version of a central conjecture from symplectic topology, the so-called nearby Lagrangian conjecture. This conjecture claims that the group of Hamiltonian diffeomorphisms acts transitively on the set of exact Lagrangians of $T^*N$. We here prove that this holds for the completion for the spectral metric of this group. As a consequence, any statement concerning the spectral norm that holds for Lagrangians that are Hamiltonian isotopic to the zero section holds for all exact Lagrangians.

\medskip
We now state more precisely our results.

We let $(M,\omega)$ be an exact symplectic manifold of dimension $2n$ and $\lambda$ a Liouville form, i.e. a $1$-form such that $d\lambda=-\omega$.
We also assume that $(M,\omega)$ is Liouville:  
there exists a sequence of compact subsets with smooth boundary $K_1\subset K_2\subset \dots $ with $M=\bigcup_{i=1}^\infty K_i$ and such that for each $i$, the Liouville vector field $X$ (defined by $\iota_X\omega=\lambda$) is transverse to $\partial K_i$ and points inward.
  The standard example is given by a Liouville domain, that is   
 a compact exact symplectic manifold $(W, -d\lambda)$ with boundary $\partial W$, such that the Liouville vector field defined by $i_X\omega=\lambda$ is transverse to the boundary. We can then extend $W$ to $W\cup \partial W\times [1,+\infty[$ with symplectic form on $\partial W\times [1,+\infty[$ given by 
$-d(t\lambda)$. Then $X$ extends to the vector field $- \frac{\partial}{\partial t}$ on $\partial W\times [1,+\infty[$ and its flow is complete\footnote{We shall often need this completeness, not only in forward time, but also in backward time. It follows from the above description that this is always the case for this extension of a Liouville domain.}. This case includes in particular the cotangent bundle $(T^*N, -d\lambda)$ of any closed smooth manifold $N$ with the standard Liouville form $\lambda=p\,d\!q$.

A \emph{conformally symplectic} diffeomorphism in $(M,\omega)$ is a diffeomorphism $\phi$ such that $\phi^*\omega=a\, \omega$ for some $a>0$. Note that in dimension $2$, any map $\phi$ satisfies $\phi^*\omega=a(z)\omega$ for some  function $a$ but we say that $\phi$ is conformally symplectic only if $a$ is constant. In higher dimension $\phi^*\omega=a(z)\omega$ implies $a(z)$ is constant, so the terminology ``conformally symplectic'' is unambiguous. Also note that if $a\neq 1$, then $M$ must be a non-compact manifold of infinite volume, and refer to 
\cite{Arnaud-Fejoz} for recent results on this topic.  The map $\phi$ is said \emph{conformally exact symplectic} if $\phi^*{\lambda}-a \lambda$ is exact
for some Liouville form $\lambda$ (i.e. satisfying $\omega=-d\lambda$). 
According to Appendix B of \cite{Arnaud-Fejoz}, if $\phi$ is homotopic to identity there exists a primitive $\lambda$ of $\omega$ such that $\phi$ is conformally exact symplectic for $\lambda$. Moreover, under mild assumptions at infinity, namely that the $\omega$-dual vector field of $\phi^* \lambda-a \lambda$ is complete (see \cite[Appendix B, prop. 10]{Arnaud-Fejoz}), $\phi$ is conjugate to a conformally exact symplectic map for the original $\lambda$.

The space of closed exact Lagrangian submanifolds in $M$ is denoted by  $\fL(M,\omega)$. 
Exact conformal maps act on $\fL(M,\omega)$ and  $\fL(M,\omega)$   carries the so-called spectral norm $\gamma$ (see Section \ref{sec:preliminaries}). Its completion with respect to the spectral norm is denoted  by $\hatL(M,\omega)$. The elements of these completions have a $\gamma$-support, which is a closed subset of $M$ (see Section \ref{sec:preliminaries} and \cite{Viterbo-gammas}).

A special case of our main theorem  (See Theorem \ref{th:main-for-branes} below) is the following result.

\begin{thm}\label{th:main}
Given a conformally exact symplectic map $\phi$  with $a\neq 1$ on $T^*N$, there is a closed invariant subset $B(\phi)$  canonically associated to $\phi$ such that $B(\phi)=\gammasupp(L)$, for some $L \in \hatL(T^*N)$ which is fixed by $\phi$. 
\end{thm}

\begin{rem} 
 Sets of the form $\gammasupp(L)$ have a number of properties. In particular they are $\gamma$-coisotropic (see \cite{Viterbo-gammas}), which implies that their Hausdorff dimension is at least $n$. We shall prove more properties of the $\gamma$-support in Theorem \ref{th:cohomology}.
 \end{rem} 

The fixed point and its $\gamma$-support will be respectively denoted by $L_\infty$ and $B$, or $L_\infty(\phi)$ and $B(\phi)$ if needed.
The subset $B$ will be called \emph{generalized Birkhoff attractor} of $\phi$. This terminology is justified by the following result.

\begin{thm}\label{th:annulus-case}
  Let $\phi$ be a conformally symplectic map homotopic to identity of the annulus $\mathbb{A}=[-1,1]\times \bS^1$ with conformal ratio $a<1$ and such that $\phi(\mathbb{A})\subset (-1,1)\times\bS^1$. Then $B(\phi)$ coincides with the classical Birkhoff attractor (see \cite{Birkhoff-attractor}).
\end{thm}

We recall the definition of the classical Birkhoff attractor in Section \ref{sec:connection-Birkhoff}. We point out that even though the Birkhoff attractor is classically defined for twist maps, its definition extends to the general case. We assume for the rest of this introduction that $M=T^*N$ is the cotangent bundle of a smooth closed manifold $N$.

The topology of $B(\phi)$ is rather well understood in the case $N=\bS^1$ : $B(\phi)$ separates the annulus, is connected, but can be an indecomposable continuum\footnote{That is a connected compact metric space that cannot be decomposed as the union of two proper connected compact subsets.}(see \cite{Charpentier-1, Charpentier-2}). What can we say in general? The following result partially answers this question.

\begin{thm}\label{th:cohomology} 
  Let $L \in \hatL(T^*N)$ be such that $\gammasupp(L)$ is compact. 
  Then, the natural map $H^\ell(N)\to \bar H^\ell(\gammasupp(L)):=\varinjlim_{U\supset \gammasupp(L)} H^\ell(U)$ is injective for any $\ell \geq 0$.
\end{thm}

The above of course applies to the Birkhoff attractor of any conformally exact symplectic map provided it is compact. This can be seen as a higher dimensional equivalent of ``separating'' (which actually means separating the two ends of the annulus). Indeed, one can check that a connected set in the annulus is ``separating'' if the cohomology of $\bS^1$ injects in its cohomology. We will prove Theorem \ref{th:cohomology} in Section \ref{sec:topological-properties}. 

The following question remains open: Is $B$ always connected ? 
This was recently answered positively in the case of a cotangent bundle in \cite{AGHIV}. It is also proved that $\gammasupp(L)$ does not have to be connected when non-compact. 

\medskip
When $\phi=\phi_{H,\alpha}^1$ is the time-1 map of a damped Hamiltonian system  with friction proportional to  velocity, that is a system driven by the vector field $X_{H, \alpha}$ such that $i_{X_{H, \alpha}}=-dH+\alpha\lambda$ and given by the equations
$$
\begin{cases}
\dot q(t)=& -\frac{\partial H}{\partial p}(t,q(t),p(t))\\
\dot p(t)=& \frac{\partial H}{\partial q}(t,q(t),p(t))-\alpha p
\end{cases}
$$
(which is conformally symplectic of ratio $a=e^{-\alpha}$) a classical approach to finding closed of invariant subsets $\phi_{-H,\alpha}^t$ consists in studying the discounted Hamilton-Jacobi equation
\begin{equation}
\label{eq:HJ-discounted}
\alpha u(x) + H(x,du(x))=0.
 \end{equation}
For given $\alpha >0$, and for $H$ coercive such an equation has a viscosity solution, and this solution is unique (see \cite{Lions1982, Barles} in the case of ${\mathbb R}^n$, but the general case is proved similarly).  
We denote it by $u_{H,\alpha}$.

We prove in Section \ref{sec:HJ}:

\begin{thm}\label{th:discounted}
  For any Tonelli Hamiltonian $H:T^*N\to\R$, consider the viscosity solution $u_{H,\alpha}$, and $x$ a point of differentiability of $u_{H, \alpha}$. Then $(x,du_{H,\alpha}(x))$ belongs to the generalized Birkhoff attractor of $\phi_{-H,\alpha}^1$.
\end{thm}

\begin{rems} 
\begin{enumerate}
\item The literature on discounted Hamilton-Jacobi equations and weak KAM theory uses a different convention. As we do here, symplectic geometers tend to use the convention $\iota_{X_H}\omega=-dH$, while dynamicists use $\iota_{X_H}\omega=dH$. The latter gives the standard sign in Hamilton equations $\dot q=\partial_pH, \dot p=-\partial_qH$, while the former is often preferred because it yields spectral invariants that are non-decreasing with respect to the Hamiltonian function. 
With the dynamicist's convention, Theorem \ref{th:discounted} becomes: \emph{For any Tonelli Hamiltonian $H:T^*N\to\R$, consider the viscosity solution $u_{H,\alpha}$, and $x$ a point of differentiability of $u_{H, \alpha}$. Then $(x,du_{H,\alpha}(x))$ belongs to the generalized Birkhoff attractor of $\phi_{H,\alpha}^1$.} 
    \item We remind the reader that a Tonelli Hamiltonian is a Hamiltonian such that
    \begin{itemize}
             \item $H$ is $C^2$ and for all $(q, p)\in T^*N$, $\frac{\partial^2 H}{\partial p^2}(q,p)$ is positive definite,
       \item $H$ is superlinear i.e. for all $q\in N$, $\lim_{\| p\|\to\infty} \frac{H(q, p)}{\| p\|}=+\infty$, where $\|\cdot\|$ is any Riemannian norm on $N$.
    \end{itemize}
\item 
In Appendix \ref{appendix}, Maxime Zavidovique constructs an example of a non-Tonelli Ha\-miltonian for which the conclusion of the above theorem does not hold, as well as an example of a Tonelli Hamiltonian for which the Birkhoff attractor of  $\phi_{-H,\alpha}^1$ is strictly larger than the closure of $\bigcup_{t\in\R}\phi_{-H,\alpha}^t(\mathrm{graph}(du_{H, \alpha}))$.
\end{enumerate}
\end{rems} 
An intermediate step in the proof of Theorem \ref{th:discounted} consists in showing that the viscosity solution coincides with the graph selector of the fixed point $L_\infty(\phi)$ in $\hatL(T^*N)$. See Section \ref{sec:HJ}.

\medskip
We also study the limit $\alpha \to 0$  or equivalently $a\to 1$. By compactness, the Birkhoff attractors of $\phi_{-H,\alpha}^1$ admit a limit point as $\alpha\to 0$, which provides a closed invariant subset of the Hamiltonian map $\phi_{-H}^1$. Moreover,
Davini, Fathi, Ituriaga and Zavidovique (\cite{DFIZ}) have shown that for $H$ Tonelli, as $\alpha$ converges to $0$, the solution $u_\alpha$ of (\ref{eq:HJ-discounted}) converges uniformly to some $u_0$.
It is then natural to ask whether the same holds for the fixed point $L_\infty(\phi_{-H,\alpha}^1)$, with respect to the $\gamma$-topology. It turns out that this is not true in the case of a pendulum  even though it is a uniformly bounded sequence with respect to the distance $\gamma$.

\begin{thm}\label{th:pendulum} If $H:T^*\bS^1\to\R$ is the Hamiltonian for a standard pendulum given by  $H(\theta,p)=\frac12p^2-\cos\theta$, then $L_\infty(\phi_{-H,\alpha}^1)$ does not admit any limit point as $\alpha$ goes to $0$ in $\hatL(T^*\bS^1)$.
\end{thm}

For the proof, we refer to Section \ref{sec:alpha-to-1}.

\medskip
In a second more symplectic topological part of the paper,  we prove on one hand Theorem \ref{th:cohomology} (see Section \ref{sec:topological-properties}), as well as a weak version of the Nearby Lagrangian conjecture (see Conjecture \ref{conj:NLC}). We denote by $\DHam(M,\omega)$  the set of smooth Hamiltonian diffeomorphisms.
Let $\widehat{\DHam}(M,\omega)$ 
be the completion of ${\DHam}(M,\omega)$ 
with respect to the spectral norm.

\begin{thm}\label{th:weak-NLC}
  Any smooth closed exact Lagrangian submanifold in $T^*N$ is the image of the zero section by an element of $\widehat\DHam(T^*N)$.  
\end{thm}

Heuristically this theorem tells us that to prove a statement involving the $\gamma$-distance for exact Lagrangians in $T^*N$, it is enough to prove it for Lagrangians Hamiltonianly isotopic to the zero section. In Section \ref{sec:weak-NLC} we prove that this holds in any Liouville manifold; (see Theorem \ref{Theorem-A4}).

\subsection*{Acknowledgments} 
We thank Maxime Zavidovique for many useful discussions and for his appendix. We are grateful to Tomohiro Asano for pointing out a missing assumption in Theorem \ref{th:cohomology} in a former version of the paper. We also thank the members of the ANR project CoSyDy for listening to preliminary versions of the results presented here and for related discussions. Finally, we are also very grateful to the anonymous referee for a careful reading and many useful comments to improve the paper.

\section{Preliminaries on \texorpdfstring{$\gamma$}{gamma} and its completion}\label{sec:preliminaries}

\subsection{Basic definitions and notation}

Any compactly supported smooth Hamiltonian $H:\mathbb{S}^1\times M\to \R$ generates a Hamiltonian isotopy $\phi_H=(\phi_H^t)_{t\in\R}$ obtained by integrating the time-dependent vector field $X_{H_t}$ which is defined by $\iota_{X_{H_t}}\omega=-dH_t$, where we use the notation $H_t(x)=H(t,x)$. The group of compactly supported Hamiltonian diffeomorphisms, i.e. the set of diffeomorphisms generated this way, will be denoted by $\DHam_c(M,\omega)$. 

A \emph{conformally symplectic} (CS) diffeomorphism is a diffeomorphism $\phi$ for which there is a constant $a\in\R$, called the \emph{conformal ratio} and such that
\[\phi^*\omega=a\,\omega.\]
We are specifically interested in the case where $a\neq 1$ and we will most of the time assume $a<1$. A conformally symplectic diffeomorphism is called \emph{exact} (CES) if $f^*\lambda-a\,\lambda$ is an exact $1$-form. It is called \emph{Hamiltonian} if it is the time-1 map of a time-dependent vector field $X_t$ such that $\iota_{X_t}\omega=\alpha_t\lambda-dH_t$, for some time dependent $\alpha_t\in\R$. For instance, the Liouville vector field $X$ is Hamiltonian (with $H_t=0$ and $\alpha_t=1$). We denote by $\phi_{H, \alpha}^t$ the isotopy generated by this vector field.

\begin{rem}\label{rem:CSCES}
  One can easily check that Hamiltonian CS diffeomorphisms are CES. As already mentioned in the introduction, it is proved in the appendix of Arnaud-Fejoz \cite{Arnaud-Fejoz} that under the assumptions specified there, any CS diffeomorphism which is homotopic to identity is CES for some Liouville form, and furthermore that given a Liouville form $\lambda$ any CS diffeomorphism which is homotopic to identity is smoothly conjugate to a CES diffeomorphism with respect to $\lambda$. 
In the sequel we shall only deal with CES maps, unless otherwise stated, and let the reader apply the result of \cite{Arnaud-Fejoz} to extend the results to the CS case.
\end{rem}

 A Lagrangian submanifold $L$ is called \emph{exact} if the restriction $\lambda|_L$ is exact. In this case, there exists a primitive function $f_L:L\to\R$ of $\lambda$ on $L$, i.e. $\lambda|_L=df_L$. If $L$ is connected, the primitive $f_L$ is unique up to addition of a constant. 

 A \emph{Lagrangian brane} of $L$ is a triple $(L,f_L,\tilde{G}_L)$ where $L$ is an exact Lagrangian submanifold, $f_L$ is a primitive for $\lambda_{\mid L}$ and $\tilde{G}_L$ is a grading of $L$, in the sense of \cite{Seidel-graded}. More precisely, $\tilde{G}_L:L\to\tilde{\Lambda}(TM)$ is a lift of the natural map $G_L:L\to\Lambda(TM)$ to the fiberwise universal cover of the Lagrangian Grassmannian of the tangent bundle $TM$. However, we will mostly forget about the grading and denote Lagrangian branes by $\tilde L=(L,f_L)$. 
 We will say that $\tilde L$ is a brane associated to $L$ or simply a \emph{lift} of $L$. 
 
 Branes may be shifted as follows: for $\tilde L=(L,f_L)$ and $c\in\R$, we have a \emph{shift} map  $T_c$  given by $T_c(L, f_L)= (L, f_L+c)$. Moreover the natural action of $\DHam_c(M,\omega)$ on exact Lagrangian lifts to an action on branes given by $\phi_H^1(L, f_L)=(\phi_H^1(L), H\sharp f_L)$, where
 \begin{equation}
 H\sharp f_L(\phi_H^1(x))=f_L(x)+\int_{\left\{\phi_H^t(x)\right\}_{t\in[0,1]}}\lambda+H\,d\!t.\label{eq:Ham-action-on-primitive}
 \end{equation}
 
Finally, let us see how CES diffeomorphisms act on branes. 
Let $\phi$ be a CES diffeomorphism of conformal ratio $a$. Choose a function $h$ such that $\phi^*\lambda-a\lambda=dh$ and denote by $\tilde \phi$ the pair $(\phi, h)$. Then any brane $\tilde L=(L,f_L)$ has an associated brane
   \begin{equation}\label{eq:CES-on-brane}
\tilde \phi(\tilde L)=\left(\phi(L)\,,\, (a\,f_L+h_{\mid L})\circ\phi^{-1}\right).
\end{equation}
We see that this depends on the choice of $h$. For instance, if $\phi=\phi_{H,\alpha}^t$ is the time-$t$ flow of a damped Hamiltonian, then a possible choice for $h$ is
 \begin{equation}
    \label{eq:conf-Ham-h}
    h(x)=\int_{-t}^0e^{\alpha s}(\iota_{ X_{H,\alpha}}\lambda + H)\circ \phi_{H,\alpha}^{t+s}(x)ds,
\end{equation}
which gives a lift $\Phi_{H,\alpha}^t$ whose action on a brane $\tilde L=(L,f_L)$ is given by $\Phi_{H,\alpha}^t(\tilde L)=(\phi_{h,\alpha}^t(L),F_L)$ where 
\begin{equation}\label{eq:conf-Ham-on-brane}
  F_L(x)=e^{-\alpha t}f_L(\phi^{-t}_{H, \alpha}(x))+\int_{-t}^0e^{\alpha s}\big (\lambda_{\phi^s_{H, \alpha}(x)}(X_{H, \alpha}(\phi_{H, \alpha}^s(x)))+H(\phi^s_{H, \alpha}(x))\big)ds.
\end{equation}

 \subsection{Lagrangian spectral invariants and the Lagrangian \texorpdfstring{$\gamma$}{gamma}-distance}\label{sec:Lagrangian-spectral-inv}

Since our manifold is convex at infinity, the Floer cohomology of any pair of closed connected 
exact Lagran\-gian submanifolds $L,L'$ is well defined (see \cite{Floer-Lag-intersection, McDuff-contact-boundaries, Viterbo-FunctorsI}) and will be denoted by $HF^*(L,L')$. The Lagrangians $L, L'$ are said to be \emph{Floer theoretic equivalent} and we write $L\sim L'$, if $HF^*(L,L')\simeq 
  H^*(L)\simeq H^*(L')$ where the isomorphisms are induced by multiplication, i.e. there exist $\alpha \in HF^*(L',L)$ and $\beta \in HF^*(L,L')$ such that  $\alpha \cup\cdot : HF^*(L,L') \longrightarrow HF^*(L',L')=H^*(L')$ and $\beta \cup \cdot : HF^*(L',L) \longrightarrow HF^*(L,L)=H^*(L)$ are isomorphisms. In particular $\alpha\cup \beta =1\in HF^*(L',L')=H^*(L')$ and $\beta \cup \alpha=1 \in HF^*(L,L)$ (see \cite[Def. 3.5]{Abouzaid-Kragh}). 

\begin{defn}   We let $I(M,\omega)$ denote the set of equivalence classes of connected exact Lagrangians for the Floer theoretic equivalence relation.
\end{defn}

The condition $L\sim L'$ ensures that $\gamma(L,L')$ is well-defined (see for example \cite{shelukhin-sc-viterbo}). If $\fL\in I(M,\omega)$ is an equivalence class, we denote by $\LL$ the corresponding set of branes. Note that any $\fL\in I(M,\omega)$ is stable under Hamiltonian isotopy but could in general be bigger than the Hamiltonian isotopy class. In the case of a cotangent bundle $(T^*N, \omega)$ it is known  \cite{Fukaya-Seidel-Smith, Kragh3} that there exists a unique equivalence class, which we denote by $\fL(T^*N)$. 

A conjecture usually attributed to Arnold (and first published in \cite{Lalonde-Sikorav}) states
 \begin{Conjecture}[Nearby Lagrangian conjecture]\label{conj:NLC}
 Let $L$ be an exact Lagrangian in $T^{*}N$, where $L,N$ are closed manifolds. Then there exists a Hamiltonian isotopy such that $L=\phi (0_{N})$. 
 \end{Conjecture}
  In dimension $4$, some partial results towards the conjecture are known (\cite{Eliashberg-Polterovich-local-2-knots, Hind, Kim-unknotedness, Dimitroglou-nearby-lagrangian}), but nothing is known in higher dimensions. 
  This conjecture implies that the projection of $L$ on $N$ is a homotopy equivalence and that we have an isomorphism $HF^{*}(L,0_{N})\simeq H^{*}(N)$. These have been proved independently (see \cite{Fukaya-Seidel-Smith, Kragh-Abouzaid}). As stated in the introduction (Theorem \ref{th:weak-NLC}), we establish a weak version of this conjecture in Section \ref{sec:weak-NLC}.
  
 Recall that if $L$ and $L'$ are transverse, the Floer cohomology is defined from a cochain complex whose underlying module is freely generated (over a ring $\mathbb{A}$) by the intersection points of $L$ and $L'$:
 \[CF(L,L'):=\bigoplus_{x\in L\cap L'}\mathbb{A}x.\]
 
 Let now $\tilde L=(L,f_L)$, $\tilde L'=(L',f_{L'})$ be branes associated to Lagrangians $L, L'$. Then, the Floer complex may be filtered by the action of intersection points. Namely, denoting the \emph{action} by
 \[\mathcal A_{\tilde L,\tilde L'}(x):=f_{L}(x)-f_{L'}(x),\]
 we have for any real number $a$ a subcomplex
 \[CF_{\geq a}(\tilde L,\tilde L'):=\bigoplus_{x\in L\cap L', \mathcal  A_{\tilde L,\tilde L'}(x)\geq a}\mathbb{A}x.\]
 The homology of the quotient \[CF_{<a}(\tilde L,\tilde L')=CF(\tilde L,\tilde L')\,\left/\,CF_{\geq a}(\tilde L,\tilde L')\right.,\] will be denoted $HF_{a}(\tilde L,\tilde L')$ and called \emph{filtered Floer cohomology} of $(\tilde L,\tilde L')$. The inclusion of complexes induces a map $i_a:HF(L,L')\to HF_a(\tilde L,\tilde L')$. 
\begin{rem}
If $L$ and $L'$ are not transverse, $HF(L,L')$ is defined as $HF(L,L'')$ for any sufficiently small Hamiltonian deformation $L''$ of $L'$ which is transverse to $L$. Then $HF(L,L'')$ has a limit as $L''$ converges  (in the $C^\infty$ topology) and remains transverse to $L'$, and this is denoted by $HF(L,L')$. It does not depend on the deformation (see e.g. \cite{Seidel}).
\end{rem}

 \begin{rem}(Generating functions)\label{rem:generating-functions} If $L\in\fL(T^*N, \omega)$ is Hamiltonian isotopic to the zero section, then it admits a generating function quadratic at infinity \cite{Laudenbach-Sikorav}, i.e. a function $S:N\times\R^d\to\R$ such that $S(x,\xi)$ coincides with a non-degenerate quadratic form $Q:\R^d\to\R$ outside a compact set and
   \[L=\left\{(x,p)\in T^*N: \exists \xi\in\R^d,\partial_xS(x,\xi)=p,\,\partial_\xi S(x,\xi)=0\right\}.\]
The choice of $S$ determines a choice of brane $\tilde L$ for which we have  canonical isomorphisms \cite{Viterbo-FCFH2, Milinkovic-equivalence, Milinkovic-Oh} 
\[HF_a(\tilde L,0_N)\simeq H^*(S^a,S^{-\infty})\]
for any $a\in\R$ and where $S^{a}=\{(x,\xi):S(x,\xi)<a\}$ and  $S^{-\infty}$ denotes $S^{<A}$ for $A<<0$.
 \end{rem}
   
For any transverse pair $(L, L')$ in the same class $\fL\in I(M,\omega)$ we may now define the spectral invariants of any corresponding branes $\tilde L$ and $\tilde L'$. For any non-zero cohomology class $\alpha\in H^*(L)$, we set
\[\ell(\alpha;\tilde L,\tilde L')=\inf\left\{a\in\R : i_a(\alpha)\neq 0\right\}.\]
Spectral invariants were first introduced and studied by Viterbo \cite{Viterbo-STAGGF} using generating functions, and then extended to more general situations using Floer theory by Oh \cite{Oh-action-functional-I}, Schwarz \cite{Schwarz}, Leclercq \cite{Leclercq}. In \cite{Leclercq}, even though he assumes $L'=\phi(L)$ it is clear that the construction extends to the general case when $L, L'$ are Floer theoretic equivalent. 
It turns out that the map $\ell(\alpha, \cdot, \cdot)$ extends to all pairs of Floer theoretic equivalent branes $(\tilde L,\tilde L')\in\LL\times \LL$ (not necessarily associated to transverse Lagrangians). Moreover, these invariants have the following properties:

 \begin{prop} Let $\tilde L,\tilde L'\in\LL$ and $\alpha\in H^*(L)\setminus\{0\}$. Then,
   \label{prop:lag-spec-inv-properties}
   \begin{enumerate}
   \item \textsc{(Spectrality)} There exists $x\in L\cap L'$, such that $\ell(\alpha;\tilde L,\tilde L')=\mathcal A_{\tilde L,\tilde L'}(x)$.
   \item \textsc{(Monotonicity and Hofer continuity\footnote{We recall that the Hofer norm of $\phi$ is the infimum of  $\int_0^1 \left[\max_{x}H_t(x)-\min_{x}H_t(x)\right]dt$ over all Hamiltonians $H$ generating $\phi$ at time one. Property (2) in Proposition \ref{prop:lag-spec-inv-properties} implies that the spectral norm $\gamma$ is continuous with respect to Hofer norm.} )} For any Hamiltonian $H\in C^\infty_c(\mathbb{S}^1\times M)$ the following inequalities hold:
     \[ \int_0^1\min_{x\in M} H_t(x) dt\leq \ell(\alpha;\phi_H^1(\tilde L),\tilde L')-\ell(\alpha;\tilde L,\tilde L')\leq \int_0^1\max_{x\in M} H_t(x) dt\]
   \item \textsc{(Shift)} For any constant $c\in\R$, we have:
     \[\ell(\alpha;\tilde L+c,\tilde L')=\ell(\alpha;\tilde L,\tilde L'-c)=\ell(\alpha;\tilde L,\tilde L')+c.\]
   \item \textsc{(Conformal invariance)} For any CES diffeomorphism $\phi$ of conformal ratio $a$, we have 
     \[\ell(\alpha;\tilde \phi(\tilde L),\tilde \phi(\tilde L'))=a\, \ell(\phi^*\alpha; \tilde L, \tilde L').\]
     In particular, this does not depend on the choice of lift $\tilde \phi$ of $\phi$.
   \item \textsc{(Triangle inequality)} For any third Lagrangian brane $\tilde L''\in\LL$, and any class $\beta$ such that $\alpha\cup\beta\neq 0$, we have
     \[\ell(\alpha\cup\beta;\tilde L,\tilde L'')\geq \ell(\alpha;\tilde L,\tilde L')+\ell(\beta;\tilde L',\tilde L'').\]
   \end{enumerate}
 \end{prop}
 
 We may now define the spectral distance as in \cite{Viterbo-STAGGF} and \cite{Viterbo-gammas}\footnote{More precisely, \cite{Viterbo-STAGGF} defines $\gamma$ via generating functions and \cite{Viterbo-gammas} defines $c$ via sheaves but these are equivalent to the Floer theoretic version we give  according to   \cite{Milinkovic-Oh} and \cite{Viterbo-Sheaves} respectively.}. 
 Given two equivalent Lagrangians $L, L'$, and two choices of respective branes $\tilde L$, $\tilde L'$, we set
 \begin{equation*}
   c(\tilde L,\tilde L')=\max\{0, \ell(\mu; \tilde L, \tilde L')\}- \min \{0,\ell(1; \tilde L,\tilde L')\}
 \end{equation*}
 where $\mu$ and $1$ respectively denote top-degree and $0$-degree generators of $HF^*(L)$. We also define
 \begin{equation*}
   \gamma(L,L')=\ell(\mu; \tilde L, \tilde L')-\ell(1; \tilde L,\tilde L').
 \end{equation*}
 Note that by the shift property in Proposition \ref{prop:lag-spec-inv-properties}, the real number $\gamma(L,L')$ does not depend on the choice of branes $\tilde L, \tilde L'$. Moreover, we have $\gamma(L, L')=\inf c(\tilde L, \tilde L')$ where the infimum runs over all branes $\tilde L$, $\tilde L'$ of $L$, $L'$ (\cite[Prop 5.2]{Viterbo-gammas})

The proposition below follows immediately from the fourth item of Proposition \ref{prop:lag-spec-inv-properties} and plays an important role in our story.
 
 \begin{prop}\label{prop:gamma} Let $\fL\in I(M, d\lambda)$. Then, $c$ is a distance function on $\LL$ and $\gamma$ is a distance function on $\fL$. Moreover, for any CES diffeomorphism $\phi$ of ratio $a$, any $L, L'\in\fL$ and any lifts $\tilde \phi, \tilde L, \tilde L'$ of these objects, we have
   \begin{equation*}
     c(\tilde \phi (\tilde L), \tilde \phi(\tilde L'))=a\, c(\tilde L,\tilde L')\quad \text{and}\quad \gamma(\phi(L), \phi(L'))=a\,\gamma(L, L').
   \end{equation*}
 \end{prop}

 The distance $\gamma$ is called the \emph{spectral distance}. The metric spaces $(\fL,\gamma)$ and $(\LL,c)$ are not complete (and not even Polish, see \cite[Prop A.1]{Viterbo-gammas}).
 We denote their respective completions by $\hatL$ and $\hatLL$. Their study was initiated in \cite{Humiliere-completion} (in their Hamiltonian version in $\R^{2n}$) and pushed further in \cite{Viterbo-gammas}.

 \begin{rem}\label{rem:completions} The following operations extend to completions in a very natural way. All these operations obviously still satisfy properties (2)-(5) from Proposition \ref{prop:lag-spec-inv-properties}.
   \begin{enumerate}
   \item The natural projection $u:\LL\to\fL$, $(L, f_L)\mapsto L$  is $1$-Lipschitz hence extends to a map $u:\hatLL\to\hatL$. This extension is surjective  \cite[Prop 5.5]{Viterbo-gammas}. 
   \item The group $\DHam_c(M,\omega)$ acts by isometry on $\fL$ and $\LL$. Therefore, this extends to actions of $\DHam_c(M,\omega)$ on $\hatL$ and $\hatLL$. We will use the same notation for these actions as before taking completion; namely we will write $\phi_H^1(L)$ and $\phi_H^1(\tilde L)$.
\item  By Proposition \ref{prop:gamma}, any CES diffeomorphism of ratio $a$ acts as an $a$-Lipschitz map on $\fL$ hence extends to a self-map of $\hatL$.  
   \item The shift map $T_c$ also acts as an isometry on $\LL$, hence extends to a map $T_c:\hatLL\to\hatLL$.
   \item For any class $\alpha$, the spectral invariant $\ell(\alpha;\cdot, \cdot)$ is Lipschitz on $\LL\times\LL$ hence extends to $\hatLL\times\hatLL$.  
\end{enumerate}
\end{rem}
  \begin{example} 
    Let $f$ be a smooth function on a closed connected manifold $N$. If $L=\Gamma_f$ is the graph of $df$, $\mu_N, 1_N$ are the generators of $H^n(N), H^0(N)$, then \[c(\mu_N,\Gamma_{f})=\max_{x\in N}f(x),\quad c(1_N,\Gamma_{f})=\min_{x\in N}f(x),\] and the other $c(\alpha,\Gamma_{f})$ are critical values of $f$ obtained by minmax on the cohomology class $\alpha$, i.e. setting $f^c=\{x\in N \mid f(x)\leq c\}$, 
    $$c(\alpha,\Gamma_f)=\inf\{c : \alpha \neq 0\;\text{in}\; H^*(f^c)\}$$
  \end{example}

 \subsection{The \texorpdfstring{$\gamma$}{gamma}-support}\label{sec:gamma-support}

We fix a given class $\fL\in I(M,\omega)$ throughout this section. The elements of the completion $\hatL$ are very abstract objects. Indeed, by definition, they are certain equivalence classes of Cauchy sequences of Lagrangian submanifolds with respect to the spectral distance. The $\gamma$-support addresses this issue by associating to any element in $\hatL$ a closed subset of $M$.

\begin{defn}[$\gamma$-support, \cite{Viterbo-gammas}]\label{defgamma} Let $L\in \hatL$. The \emph{$\gamma$-support} of $L$ is the set of all $x\in M$ such that for any neighborhood $U$ of $x$, there exists $\phi\in\DHam_c(M,\omega)$ supported in $U$ which satisfies \[\phi(L)\neq L.\] 
\end{defn}

In other words, a point $x$ does not belong to $\gammasupp(L)$ if any Hamiltonian diffeomorphism supported in a sufficiently narrow neighborhood of $x$ leaves $L$ invariant. Note that this is analogous to the definition of support for 
distributions, where compactly supported Hamiltonian diffeomorphisms play the role of test functions.

We will also need the following 
\begin{defn}[\cite{Usher-rigidity,Viterbo-gammas}]
  A locally closed subset $V$ in a symplectic manifold  is \emph{$\gamma$-coisotropic} if for any $x\in V$ there exists a ball $B$ centered at $x$ such that for any smaller ball $B'\subsetneq B$ centered at $x$, there is $\delta>0$ such that for any $\phi\in\DHam_c(M,\omega)$ supported in $B$, 
    \[\phi(V)\cap B'=\emptyset\quad\implies\quad  \gamma(\phi)>\delta.\]
\end{defn}
From \cite{Viterbo-gammas}, we may assert that if $V$ is $\gamma$-coisotopic, then it has Hausdorff dimension at least $n$, and that   a {\bf closed} submanifold   is $\gamma$-coisotropic if and only if  it is coisotropic in the usual sense.  
The $\gamma$-support is well behaved in many respects. 
We refer to \cite{Viterbo-gammas} for the following proposition which lists some important properties of $\gamma$-supports.

\begin{prop}\label{prop:gamma-support}
We have
  \begin{enumerate}
\item \textsc{(Regular Lagrangians)} For any smooth Lagrangian  $L\in\fL$, we have $\gammasupp(L)=L$.
  \item \textsc{(Invariance)} For any CES diffeomorphism $\psi$ (e.g. if $\psi\in\DHam_c(M,\omega)$), and for any $L\in\hatL$, we have
    \begin{equation*}
      \gammasupp(\psi(L))=\psi(\gammasupp(L)).
    \end{equation*}
  \item \label{Prop-2.13-3}\textsc{ ($\gamma$-coisotropic)} For any $L\in\fL$,  $\gammasupp(L)$ is $\gamma$-coisotropic 
    
  \item \textsc{(Semi-continuity)} For any open subset $U$ and any sequence $L_k\in\hatL$, $k\geq0$, with  $\gammasupp L_k\subset U$ which $\gamma$-converges to some $L\in\hatL$ for all $k\geq0$, then $\gammasupp(L)\subset U$. In other words, we have
    \[\gammasupp(L)\subset\bigcap_{k_0\geq0}\overline{\bigcup_{k\geq k_0}\gammasupp(L_k)}.\]
  \end{enumerate}
\end{prop}

\begin{rem}
\begin{enumerate}
    \item  The $\gamma$-support does not determine $L\in \hatL$, unless the $\gamma$-support is an $n$-dimensional manifold (\cite{AGIV}).
    \item It is not known whether for a given $L\in\hatL$ and a neighborhood $U$ of $\gammasupp(L)$, we can always find a sequence $(L_k)_{k\geq 1}$ converging to $L$ and contained in $U$. 
\end{enumerate}
\end{rem}

\begin{rem}\label{rem:sympeo} Homeomorphisms which are $C^0$-limits of Hamiltonian diffeomorphisms are sometimes called \emph{Hamiltonian homeomorphism}. Since the $\gamma$-norm on $\DHam_c(M,\omega)$ is $C^0$ continuous\footnote{This follows from \cite{Viterbo-STAGGF} on $\mathbb R^{2n}$ and \cite{BHS-C0} for the general case.}, the natural action of Hamiltonian diffeomorphisms on $\hatL$ extends to an action of Hamiltonian homeomorphisms on $\hatL$. It then follows from Proposition \ref{prop:gamma-support} items (2),(4), that 
 \begin{equation*}
      \gammasupp(\psi(L))=\psi(\gammasupp(L)).
    \end{equation*}
  holds for any Hamiltonian homeomorphism $\psi$.
\end{rem}

The $\gamma$-support has many more properties. For instance, it is known that in a cotangent bundle $T^*N$, it has non-trivial intersection with all fibers as well as with every closed exact Lagrangian submanifolds.
We refer the reader to \cite{Viterbo-gammas} for more details on $\gamma$-supports.

We will also need the following lemma.

\begin{lem}[\cite{Viterbo-gammas}, Lemma 6.12]\label{lem:shift-of brane} Let $\tilde L\in \hatLL$ correspond to an element $L\in\hatL$, and let $H$ be a Hamiltonian for which there exists $C\in\R$ such that $H(t,x)=C$ for any $t\in\mathbb{S}^1$, $x\in \gammasupp(L)$. Then $\phi_H^1$ acts on $\tilde L$ as a shift: 
  \[\phi_H^1(\tilde L)=T_C\tilde L.\]
\end{lem}

Note that in the case $\tilde L\in\LL$ is a genuine smooth brane, this would follow immediately from (\ref{eq:Ham-action-on-primitive}).

\section{The generalized Birkhoff attractor}\label{Section-Birkhoff-attractor}

\subsection{Proof of Theorem \ref{th:main} and first properties}

In this section we assume that $\phi$ is a conformally exact symplectic diffeomorphism with $a\neq 1$, on some Liouville manifold $(M,\omega=-d\lambda)$.

Our first task will be to prove a generalization of Theorem \ref{th:main} to such manifolds. In fact we will prove a refined version which applies to Lagrangian branes and obviously implies Theorem \ref{th:main}. We will then end the section with some extra properties and remarks.

Recall that $\phi$ not only acts on $\hatL(M,\omega)$, but also on the completion of Lagrangian branes $\hatLL(M,\omega)$ with respect to the metric $c$, by Remark \ref{rem:completions}.

\begin{thm}\label{th:main-for-branes} 
Assume that $(M,\omega)$ admits a non-empty class of closed exact Lagrangian submanifolds $\fL\in I(M,\omega)$ which is preserved by $\phi$ and denote by $\LL$ the corresponding class of branes. 
Then for any lift $\tilde\phi$ of $\phi$, there is a unique element $\tilde L_\infty$ in the $c$-completion $\hatLL$
such that $\tilde\phi\left(\tilde L_\infty\right)=\tilde L_\infty$. As a consequence, denoting by $L_\infty$ the element of $\hatL$ corresponding to $\tilde L_\infty$,
the subset \[B(\phi)=\gammasupp(L_\infty)\] is an invariant, closed and $\gamma$-coisotropic subset of $\phi$.
\end{thm}

This result obviously implies Theorem \ref{th:main}.

\begin{proof}[Proof of Theorem \ref{th:main-for-branes}] 
Changing $\phi$ to $\phi^{-1}$ we may assume $a<1$. Then $\tilde\phi$ acts on $\LL(M, \omega)$ as a contraction, by Proposition \ref{prop:gamma}. 
By Picard's fixed point theorem, $\tilde\phi$ has a unique fixed point $\tilde L_\infty \in
\hatLL(M, \omega)$. By Proposition \ref{prop:gamma-support}, item (2), we must then have
$\phi(\gammasupp(L_\infty))= \gammasupp(L_\infty)$. Finally $B(\phi)$ is $\gamma$-coisotropic by Proposition \ref{prop:gamma-support}, (\ref{Prop-2.13-3}). 
\end{proof}

We also have
\begin{prop} \label{prop:lagrangian-in-B}
Assume that $B(\phi)$ contains a closed exact Lagrangian
submanifold $\Lambda$. Then the union of the images  $\phi^k(\Lambda)$
(for $k\geq k_0$) is dense in $B(\phi )$.
\end{prop}

  \begin{rem}
    It is interesting to compare the above result with Question 3 studied in Appendix \ref{appendix}.
  \end{rem}

  \begin{proof} In this situation, the  fixed point $L_\infty(\phi )$ must be the $\gamma$-limit of $\phi^k(\Lambda)$. But then, by the fourth item of Proposition \ref{prop:gamma-support}, $$B(\phi)=\gammasupp(L_\infty)\subset
     \bigcap_{k_0}\overline{\bigcup_{k\geq k_0} \gammasupp(\phi^k(\Lambda))}= \bigcap_{k_0}\overline{\bigcup_{k\geq k_0}\phi^k(\Lambda)}.$$
Since on the other hand $\bigcup_{k\geq k_0}\phi^k(\Lambda)\subset B(\phi)$ for any $k_0$,     this implies that the union of  images  $\phi^k(\Lambda)$ (for $k\geq k_0$) is dense in $B(\phi )$.
  \end{proof}

  We can say more if the following conjecture holds

  \begin{Conjecture}[Viterbo \cite{SHT}]\label{conj-bounded-Lagrangian}
    For any $r>0$, the metric $\gamma$ is bounded on the space of exact Lagrangians which are included in a disc bundle $\{(q,p)\in\T^*N:\|p\|\leq r\}$.
  \end{Conjecture}

  This conjecture is known to hold on a large class of manifolds, for instance $\mathbb S^n$, $\mathbb T^n$, compact Lie groups, compact homogenous spaces and others \cite{Shelukhin-Zoll,shelukhin-sc-viterbo,Guillermou-Vichery, Viterbo-inverse-reduction}.
  
  \begin{cor} Assume Conjecture \ref{conj-bounded-Lagrangian} holds for $N$. If $B(\phi)$ contains a closed exact Lagrangian $\Lambda$, then $B(\phi)=\Lambda$.   
  \end{cor}
\begin{proof} 
Indeed, $\gamma(\phi^{-k}(\Lambda), L_\infty(\phi))=\gamma(\phi^{-k}(\Lambda), \phi^{-k}( L_\infty(\phi)) =a^{-k} \gamma(\Lambda, L_\infty(\phi)) $ goes to $+\infty$. But since $\Lambda \subset B(\phi)$, $\phi^{-k}(\Lambda)$ remains bounded, this contradicts the conjecture. 
\end{proof} 

  \begin{rem}
   The set
$B(\phi )$ is the minimal  invariant set of the form
$\gammasupp(\Lambda)$, $\Lambda\in \hatL(M,\omega)$. Indeed, if $\gammasupp(\Lambda) \subset B$
then $\gammasupp(\phi(\Lambda))\subset B$, since the sequence $\phi^k(\Lambda)$
$\gamma$-converges to $L_\infty$ and each element has support in $B$, we deduce that the sequence 
$ \phi^k (\gammasupp(\Lambda))$ must be dense in $B(\phi )$. 
\begin{Question}
Are there examples where we can find an invariant set strictly containing $B(\phi)$ and of the form $\gammasupp(\Lambda)$ ? 
\end{Question}
  \end{rem}

\subsection{Connection with the Birkhoff attractor for the annulus}\label{sec:connection-Birkhoff}
We refer to \cite{LeCalvez-Birkhoff1} and \cite{LeCalvez-Birkhoff2} for details on the classical Birkhoff attractor, whose construction we now recall.

Let us consider  the annulus $\mathbb A= {\mathbb S}^{1}\times [-1,1]$ supposed to be contained in the sphere $\mathbb S^{2}$ as the thickening of the equator. Let $\phi$ be  dissipative map of $\mathbb A$, i.e. there exists $\alpha<1$ such that for all measurable sets $U$, we have $\mu (\phi(U))\leq \alpha \mu(U)$. We assume that $\phi(\mathbb A)\subset {\mathbb S}^{1}\times (-1,1)$. 
Then the set $C_{0}=\bigcap_{n\geq 1}\phi^{n}(\mathbb A)$ is an invariant set, and has measure zero. As a decreasing sequence of compact connected sets, it is compact connected. Furthermore, $C_0$ is the largest compact invariant subset of $\mathbb A$.
If we set $U_{n}\cup V_{n}= \mathbb A \setminus \phi^{n}(\mathbb A)$, where $U_{n}$ is the connected component containing $\mathbb S^{1} \times \{1\}$ and $V_{n}$ the connected component containing $\mathbb S^{1} \times \{-1\}$, we have $U_{0}^+=\bigcup_{n}U_{n}, U^{-}_{0}= \bigcup_{n}V_{n}$ satisfy $U_{0}^{+}\cup U^{-}_{0}=\mathbb A\setminus C_{0}$. 

But we can find a smaller invariant set by ``cutting out the hair'' from $C_{0}$. 
In other words $C_{0}$ is a compact connected subset separating $\mathbb S^{2}$ in two simply connected sets, $U^{+}_{0},U^{-}_{0}$ such that $\mathbb S^{2} \setminus C_{0}=U^{+}_{0}\cup U^{-}_{0}$. Then if $\Fr(U^{+}_{0})=$ denotes the frontier of $U^{+}_{0}$, $C_{1}=\Fr(U^{+}_{0})\cap \Fr(U^{-}_{0})$, then $C_{1}$ is contained in $C_{0}$ and is an invariant set. It is obtained by removing the points of $C_{0}$ which are not adherent to both $U^{+}_{0}$ and $U^{-}_{0}$ (see Figure \ref{fig-8n} and  \ref{fig-9n}).  We shall denote by $U_{1}^{+}, U_{1}^{-}$ the connected components of $\mathbb A \setminus C_{1}$. 
We then have $C_{1}=\overline U_{1}^{+} \cap \overline U_{1}^{-}=Fr(U_{1}^{+})=Fr (U_{1}^{-})$. The subset $C_1$ is called the \emph{Birkhoff attractor} of $\phi$.

Since the subset $B(\phi)$ from Theorem \ref{th:main-for-branes} is a compact invariant subset, we have $B(\phi)\subset C_0$. However, $B(\phi)$ cannot be equal to $C_{0}$, because $C_{0}$ can be non $\gamma$-coisotropic at certain points  e.g. at the end of the hair ( see Figure \ref{fig-8n}) for the same reason $[0,1] \subset {\mathbb R}^{2}$ is not $\gamma$-coisotropic at $0$ or $1$.

\setlength{\fboxrule}{0.5pt}
 \begin{figure}[H]
 \begin{overpic}[width=9cm]{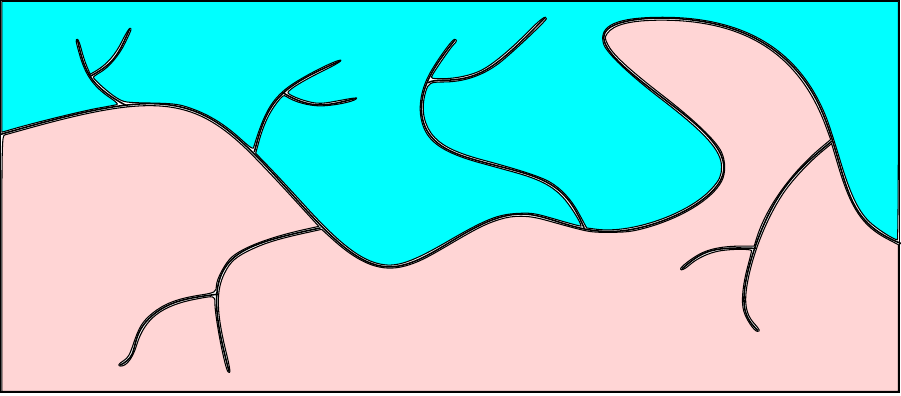}
  \put(40,35){$U^+_{0}$}
  \put(40,5){$U^{-}_{0}$}
   \put (83,6){$\bullet$}
      \put (50,39){$\bullet$}
   \put (75,13){$\bullet$}
 \end{overpic}
\caption{ The invariant set $C_{0}$ : it is not $\gamma$-coisotropic for example at the points marked ``$\bullet$''. The blue set is $U_{0}^+$, the pink set is $U^{-}_{0}$. }
\label{fig-8n}
\end{figure}

 \begin{figure}[ht]
 \begin{overpic}[width=9.5cm]{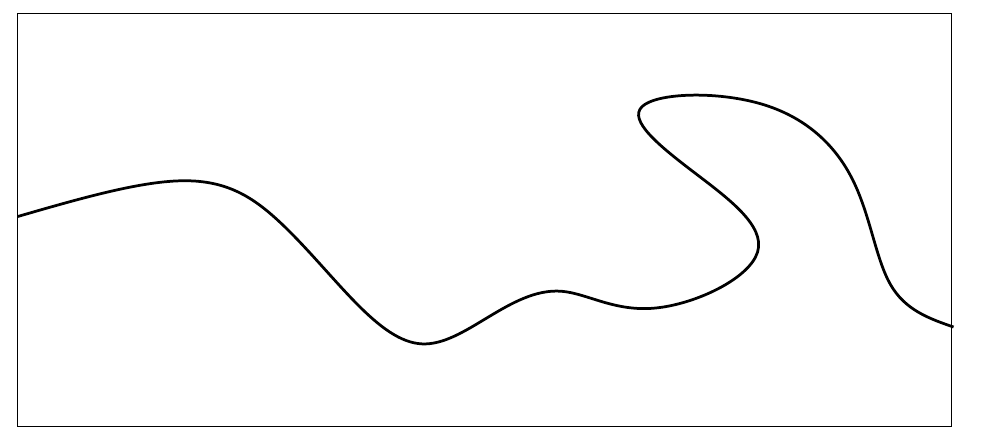}
 \end{overpic}
\caption{ The invariant set $C_{1}$   }
\label{fig-9n}
\end{figure}

We are now ready to prove Theorem \ref{th:annulus-case}, namely the fact that our invariant subset $B(\phi)$ coincides with the Birkhoff attractor $C_{1}(\phi )$.

\begin{proof}[Proof of Theorem \ref{th:annulus-case}] As a first step, we prove the inclusion
  \begin{equation}
 B(\phi) \subset C_{1}(\phi ).\label{eq:inclusion-birkhoff}
\end{equation}

  To prove this inclusion, we need to consider for $a\in[-1,1]$, the set $\mathfrak{L}_a(\mathbb A)$ of simple curves homologous to $\mathbb S^1 \times \{0\}$,  with Liouville class $a \in H^{1}(\mathbb S^{1}, {\mathbb R}) \simeq {\mathbb R} $. We have $\mathfrak L_{a}( \mathbb A) =\tau_a\mathfrak L( \mathbb A)$, where $\tau_a$ denotes the translation $(x,p)\mapsto (x,p+a)$. Note that $\phi$ sends $\mathfrak L_{a}( \mathbb A)$ to $\mathfrak L_{\alpha a}( \mathbb A)$.

  We will show that for any $\Lambda$ in $\mathfrak L_{a}( \mathbb A)$, the sequence of exact curves $\Lambda_k=\tau_{-\alpha^ka}\phi^k(\Lambda)$ converges with respect to $\gamma$ to the fixed point $L_\infty(\phi)$. By taking $\Lambda=\mathbb S^1\times \{a\}$ for $a$ close enough to $1$, we have $\Lambda\subset U_1$, and by Proposition 6.17 of \cite{Viterbo-gammas}
  $$\gammasupp (L_{\infty}(\phi )) \subset \liminf_k\Lambda_k= \liminf_k\phi^k(\Lambda)\subset \liminf_k{U_k}=\overline{U_1^+}.$$
  Similarly  $\gammasupp (L_{\infty}(\phi ))\subset\overline{U_1^-}$ and (\ref{eq:inclusion-birkhoff}) follows since $B(\phi)=\gammasupp (L_{\infty}(\phi ))$.

  Let us now prove our claim, i.e. that $\Lambda_k$ $\gamma$-converges to $L_\infty$. We will first prove that $\Lambda_k$ is a Cauchy sequence. For all $k\geq 1$, we set $f_k=\tau_{-\alpha^ka} \phi \tau_{\alpha^{k-1}a}$, so that $\Lambda_{k}=f_k(\Lambda_{k-1})$. We have that $f_{k}$ converges to $\phi$ for $\gamma$, and the $f_{k}$ (and $\phi)$ are $\alpha$-contractions on $\widehat{\mathcal L} ( \mathbb A)$. 
  
  We now have the following fixed point theorem for which we have not found any reference.
  
  \begin{prop} 
  Let $(T_{k})_{k\geq 1}$ be a sequence of maps from a complete metric space $(X,d)$ to itself. Assume that 
   \begin{enumerate} 
  \item  there exists $\alpha\in[0,1)$ such that, for all $k$ the map $T_{k}$ is $\alpha$-Lipschitz. 
  \item $(T_{k})_{k\geq 1}$ converges uniformly to a map $T_{\infty}$.
  \end{enumerate} 
  Then for any $x\in X$, the sequence $T_{k}\circ T_{k-1}\circ .... \circ T_{1}(x)$ converges to $x_{\infty}$, the unique fixed point of $T_{\infty}$.
  \end{prop} 
  
  \begin{proof} 
  We know that $T_{\infty}^{k}(x)$ converges to $x_{\infty}$ by the standard proof of Banach's fixed point theorem. Now we set 
  $$u_{k}= \sup_{x\in X} d(T_{k}\circ T_{k-1}\circ .... \circ T_{1}(x), T_\infty^{k}(x)).$$ We have
  \begin{align*}  d(T_{k}&\circ T_{k-1}\circ .... \circ T_{1}(x),T_\infty^{k}(x)) \\
    &\leq  d( T_{k}\circ T_{k-1}\circ .... \circ T_{1}(x), T_{k}(T_\infty^{k-1}(x))+d ( T_{k}(T_\infty^{k-1}(x)), T_\infty(T_\infty^{k-1}(x))) \\ &\leq 
\alpha\, d( T_{k-1}\circ .... \circ T_{1}(x),T_\infty^{k-1}(x)) + d(T_{k}, T_\infty)
  \end{align*} 
  which means, setting $ \varepsilon_{k}=d(T_{k},T_{\infty})=\sup_{x\in X}d(T_{k}(x),T_{\infty}(x))$
  $$u_{k}\leq   \alpha u_{k-1}+\varepsilon_{k}.$$
  Moreover $u_1$ is finite, since  we may always assume $d(T_1,T_\infty)<+\infty$. 
  Using the identity $$u_{n}- \alpha^{n-1}u_{1}= \sum_{j=0}^{n-2} \alpha^{j}(u_{n-j}-\alpha u_{n-(j+1)})$$ this implies that $$u_{n}\leq  \sum_{j=0}^{n-2} \alpha^{j} \varepsilon_{n-j}+  \alpha^{n-1}u_{1}.$$
 
  If $A$ is a bound for the sequence $( \varepsilon_k)_{k\geq 2}$ and for $k\geq r$ we have $ \varepsilon _{k}\leq \varepsilon $ then 
  \begin{gather*}u_{n}\leq  \sum_{j=0}^{n-2} \alpha^{j} \varepsilon_{n-j} +\alpha^{n-1}u_{1}\leq \varepsilon \sum_{j=0}^{n-r}\alpha^{j}+ M\sum_{j=n-r+1}^{n-2}\alpha^{j}\leq \frac{ \varepsilon \alpha }{1- \alpha}+  \alpha^{n-r+1} \frac{A}{1-\alpha}.
  \end{gather*} 
  Clearly this is bounded by $ \frac{2 \varepsilon }{1- \alpha}$ for $n$ large enough. 
  We thus proved that $u_{k}=d(T_{k}\circ T_{k-1}\circ .... \circ T_{1}(x),T_\infty^{k}(x))$ converges to $0$, hence $T_{k}\circ T_{k-1}\circ .... \circ T_{1}(x)$ converges to $x_{\infty}$.  
  \end{proof} 
  \begin{rem} 
  In the assumptions, it is of course sufficient to assume that the convergence from $T_{k}$ to $T_{\infty}$ is uniform on bounded sets. Indeed, we only need to bound the distance $d(T_{k}(T_{\infty}^{k-1}(x)), T_{\infty}(T_{\infty}^{k-1}(x)))$ and we know that the sequence $ T_{\infty}^{k}(x)$ is bounded. 
  \end{rem}

Applying the above Proposition to $f_{k}$ and $\phi$, we conclude that $\gamma-\lim \Lambda_{k}=\Lambda_\infty$ with $\Lambda_\infty=L_\infty$. In other words, the sequence $\Lambda_k$ converges to $L_\infty$. This proves our claim and concludes the proof of the inclusion (\ref{eq:inclusion-birkhoff}).

We now turn to the proof of the equality. By Proposition 6.10 in \cite{Viterbo-gammas}, the subset $B(\phi )=\gammasupp(L_{\infty}(\phi ))$ intersects all curves isotopic to the vertical. Therefore, it is an annular set, i.e. it separates the annulus.

So  $\mathbb A \setminus B(\phi ) = W^{+}\cup W^{-}$, the two unbounded connected components of the boundary, as there can be no bounded connected component, otherwise the union of such  components would be invariant and being open have non-zero measure. Since $B(\phi )\subset C_{1}(\phi )$, we must have 
$\mathbb A \setminus C_{1}(\phi ) \subset \mathbb A \setminus B(\phi ) $ hence $U_{1}^{+} \subset W^{+}$ and $U_{1}^{-} \subset W^{-}$. Let us prove that we have equality in both inclusions. Indeed, let $x\in  W^{+} \setminus U_{1}^{+}$. Then there is a positive $ \varepsilon $ such that $B(x, \varepsilon )\subset W^{+}$, hence $d(x,W^{-})  \geq \varepsilon $. But then $d(x,U_{1}^{-}) \geq \varepsilon $ since $U_{1}^{-} \subset W^{-}$. But if $x\notin U_{1}^{+}$ we must have 
$x\in \overline{U_{1}^{-}}$ and $d(x,U_{1}^{-})=0$ a contradiction. 
So we must have  $U_{1}^{+} = W^{+}$ and $U_{1}^{-} = W^{-}$ and we may conclude $B(\phi )=C_{1}(\phi )$.
\end{proof}

The following example of $\gamma$-support then  follows from the work of  Birkhoff (\cite{Birkhoff-attractor} and Marie Charpentier (\cite{Charpentier-1}). Remember that a continuum is a connected compact metric space. It is indecomposable if it cannot be written as the union of two
non-trivial (i.e. different from the whole space and the empty set) continua. Note that a closed curve is NOT indecomposable.

\begin{cor} \label{Cor-C7}
There exists a conformally symplectic map such that $\gammasupp(L_{\infty})$ is an indecomposable continuum.\end{cor} 
\begin{proof} Note that Birkhoff's construction in section 7 of \cite{Birkhoff-attractor} is not only dissipative (i.e. strictly reduces the areas by a factor bounded by $\alpha <1$), it is a conformally symplectic map of ratio $1- \varepsilon $ for $ \varepsilon >0$. Moreover Birkhoff proves that in this example, the Birkhoff attractor has two distinct rotation numbers. According to M. Charpentier (\cite{Charpentier-1}) this implies that $C_{1}$ is an indecomposable continuum. 
But by Theorem \ref{th:annulus-case}, this implies that $\gammasupp(L_{\infty})$ is an indecomposable continuum. 
\end{proof} 
\begin{rem} 
Even though this is quite far from the subject of this article, according to \cite{AGHIV}, the $\gamma$-support of $L_{\infty}$ coincides with the reduced singular support of its quantifying sheaf (defined in \cite{Guillermou-Viterbo}) $\F_{L_{\infty}}$ belonging to the derived category of limits of constructible sheaves. Thus there exists a limit of constructible sheaves such that its singular support in $T^{*}({\mathbb S}^{1}\times {\mathbb R} )\setminus 0_{{\mathbb S}^{1}\times {\mathbb R} }$ is an indecomposable continuum. 
\end{rem}


\section{The discounted Hamiltonian-Jacobi equation and the Birkhoff attractor}\label{sec:HJ}

The goal of this section is to prove Theorem \ref{th:discounted}. The proof will use discounted weak KAM theory and graph selectors. We introduce the relevant material from weak KAM theory in Section \ref{sec:Discounted-weak-KAM-theory} and for graph selectors in Section \ref{sec:graph-selectors}. The proof of Theorem \ref{th:discounted} is then done in Section \ref{sec:proof-discounted}.

In this section we work on a cotangent bundle $M=T^*N$, where $N$ is a closed manifold, endowed with the standard Liouville form $\lambda$. We fix a Riemannian metric on $N$ and denote by $\| p\|$ the (dual) norm of an element $p\in T^*_qN$.
We assume that $H:T^*N\to \R$  is an autonomous Tonelli Hamiltonian,  i.e. that its second order fiberwise derivative 
is positive definite and that $H(q,p)$ goes to $+\infty$ as $\|p\|\to\infty$. 

It is proven in \cite{ArnaudSuZavidovique} that if $C=\max\{H(q, 0); q\in M\}$ is the maximum of $H$ on the zero-section, then $H$ is a strict Lyapunov function on $U_H=\{ H>C\}$ for $(\phi_{-H, \alpha}^t)$. Hence the conformally Hamiltonian flow $(\phi^t_{-H, \alpha})$ is defined for all positive times. However,  it is not always complete for negative times (see also \cite{MaroSor2017}).
This implies that the compact subset
\begin{equation}\label{EKH}\mathcal K_H=T^*N\backslash U_H\end{equation}  is forward invariant by $\phi_{-H,\alpha}^t$, and contains the $\omega$-limit set of every point and for every compact subset $K$, there is a positive time such that $\phi^t_{-H, \alpha}(K)\subset \mathcal K_H$.  This last point implies that the Birkhoff attractor $B_{H, \alpha}$ of $\phi_{-H,\alpha}^t$ is compact (and contained in $\mathcal K_H$). Then, the largest invariant compact set is  $K_{H,\alpha}=\bigcap_{t>0}\phi^t_{-H, \alpha}(\mathcal K_H)$, which also contains $B_{H, \alpha}$.
 
  We recall Proposition 18 of \cite{CIPP}.
 \begin{prop} Given a Tonelli Hamiltonian $H:T^*N\to\R$  and $k\in\R$, there is a Tonelli Hamiltonian  $H_0$, convex and quadratic at infinity such that $H_0(x,p) = H(x, p)$ for every $(x, p)$ such that $H(x, p) \leq k$.
 \end{prop} 
 If $H$ is Tonelli, we  choose $k\in \R$ such that $H$ is a Lyapunov function on $\{ H\geq k\}$ and  pick $H_0$ as in the above proposition. Then $(\phi^t_{-H, \alpha})$ and $(\phi^t_{-H_0, \alpha})$ have the same Birkhoff attractor, that is contained in $\{ H\leq k\}$.

 Reminders on   discounted weak KAM theory are given in Section \ref{sec:Discounted-weak-KAM-theory} below. The weak KAM solution is also the viscosity solution of the discounted Hamilton-Jacobi equation          
$$\eqref{eq:HJ-discounted}\quad
\alpha u(x) + H(x,du(x))=0.
$$                                   
In \cite[p.38]{DFIZ},  it is explained that for $k$ large enough, the viscosity solutions of     $\alpha u(x) + H(x,du(x))=0$ and $\alpha u(x) + H_0(x,du(x))=0   $ are the same.    Hence, replacing $H$ by $H_0$, we can assume in the remainder of the article that the discounted flow of $-H$ is complete.                                                         
Moreover, if there exists a constant $C$ such that $|\partial_pH\cdot p|\leq C|H|$, then $X_{-H, \alpha}$  is complete. In both cases, there is a maximal invariant compact subset that will be denoted by $K_{H, \alpha}$. It contains the Birkhoff attractor $B_{H, \alpha}$ of $\phi_{-H,\alpha}^t$. Recall that $B_{H,\alpha}$ is by definition the $\gamma$-support of the unique fixed point $L_\infty(H, \alpha)$ of $\phi_{-H,\alpha}^t$ in $\hatL$. The flow $\phi_{-H,\alpha}^t$ determines an action $\Phi_{-H,\alpha}^t$ on the space of branes $\LL$ given by formula (\ref{eq:conf-Ham-on-brane}), which we now recall for the reader's convenience:  $\Phi^t_{-H, \alpha}(L, f_L)=(\phi^t_{-H, \alpha}(L), F_L)$
where \[F_L(z)=e^{-\alpha t}f_L(\phi^{-t}_{-H, \alpha}(z))+\int_{-t}^0e^{\alpha s}\big (\lambda_{\phi^s_{-H, \alpha}(z)}(X_{-H, \alpha}(\phi_{-H, \alpha}^s(z)))-H(\phi^s_{-H, \alpha}(z))\big)ds.\]
The unique fixed point of $\Phi_{-H,\alpha}^t$ in $\hatLL$, which is provided by Theorem \ref{th:main-for-branes}, will be denoted by $\tilde L_\infty(H, \alpha)$.  

We define the function $U_{H, \alpha}:K_{H, \alpha}\to \R$ by
$$U_{H, \alpha}(x)=\int_{-\infty}^0 e^{\alpha t}\Big( \lambda_{\phi_{-H, \alpha}^t(x)}\big(X_{-H, \alpha}(\phi_{-H, \alpha}^t(x))\big)-H\big(\phi_{-H, \alpha}^t(x)\big)\Big)dt.$$

  \subsection{Discounted weak KAM theory}\label{sec:Discounted-weak-KAM-theory}
  
  To $H$ Tonelli, we associate a Lagrangian function $L:TN\to\R$ and
  the discounted Lax-Oleinik semi-group $(T_{H, \alpha}^t)_{t\geq 0}$ is defined on the set of continuous functions $C^0(N, \R)$ by
 \begin{equation} \label{eq:DLO-semi-group} T_{H, \alpha}^tu(q)=\inf_{\gamma:[-t, 0]\to N, 0\mapsto q} \Big( e^{-\alpha t}u(\gamma(-t))+\int_{-t}^0 e^{\alpha s} L(\gamma (s), \dot \gamma (s)) ds\Big) ,\end{equation}
  where 
  the infimum is taken over all absolutely continuous curves $\gamma$ ending at $q$. Since $\|T^t_{H, \alpha}u_1-T^t_{H, \alpha}u_2\|_\infty\leq e^{-\alpha t}\| u_1-u_2\|_\infty$, the semi-group has a unique fixed point called \emph{discounted weak KAM solution} and denoted by $u_{H, \alpha}$. Moreover, for every $u\in C^0(N, \R)$, the functions $T_{H, \alpha}^tu$ converge uniformly to $u_{H, \alpha}$ as $t$ goes to $+\infty$.
  
  Then we have 
  \begin{equation} \label{eq:DLO-fixedpoint }  \forall t\geq 0, u_{H, \alpha}(q)=\inf_{\gamma:[-t, 0]\to N, 0\mapsto q} \Big( e^{-\alpha t}u_{H, \alpha}(\gamma(-t))+\int_{-t}^0 e^{\alpha s}\big(L(\gamma (s), \dot \gamma (s))\big)ds\Big) ,\end{equation}
and
 \begin{equation} \label{eq:expression-DLO-fixedpoint }  u_{H, \alpha}(q)=\inf_{\gamma:]-\infty, 0]\to N, 0\mapsto q}  \int_{-\infty}^0 e^{\alpha t}L(\gamma (t), \dot \gamma (t))dt\end{equation}
  where the infimum is taken over all absolutely continuous curves $\gamma: ]-\infty, 0]\to M$ such that $\gamma(0)=q$. 
  Moreover, there exists  a curve $\gamma_\alpha: ]-\infty, 0]\to M$ such that we have equality in \eqref{eq:DLO-fixedpoint } and \eqref{eq:expression-DLO-fixedpoint }; see \cite[Appendix B]{DFIZ}. It then follows from a classical computation that this curve is a solution of the discounted Euler-Lagrange equations
  $$\frac{d}{dt}\Big( e^{\alpha t}\partial_vL\Big)=e^{\alpha t}\partial_qL.$$

  Recall that a covector $p\in T^*_{q_0}N$ is a \emph{super-differential} of a function $u$ if there exist a local coordinate chart around $u$ and  $K>0$ such that for any $q$ we have: $u(q)-u(q_0)\leq p(q-q_0)+K\|q-q_0\|^2$.

  \begin{prop}\label{prop:super-differential} If the infimum in \eqref{eq:expression-DLO-fixedpoint } is attained at a curve $\gamma:]-\infty, 0]\to M$,    
    then the covector $\partial_vL(\gamma(0), \dot\gamma(0))$ is 
 a super-differential of $u_{H, \alpha}$ at $q=\gamma(0)$.
  \end{prop}
 
  \begin{proof} We use ideas similar to the proof of Proposition 4.11.1 in \cite{Fathi-book} (see also \cite{Bernard2008}). We choose $\gamma:]-\infty, 0]\to N$ such that the minimum in  \eqref{eq:expression-DLO-fixedpoint } is attained at $\gamma$. We work in a coordinate chart around $q$. Let $Q$ be close to $q$. 
    Fixing a large real number $A>0$, we denote by $\gamma_1:]-\infty, 0]\to N$ the arc such that $\gamma_{1|]-\infty, -1/A]}=\gamma_{|]-\infty, -1/A]}$ and in our chosen chart
    $$\forall\sigma\in [-1/A, 0],\quad \gamma_1(\sigma)=\gamma(\sigma)+(1+A\sigma)(Q-q).$$
 Note that $\gamma_1(0)=Q$. 
 We deduce from \eqref{eq:expression-DLO-fixedpoint } that $ u_{H, \alpha}(Q)\leq  \int_{-\infty}^0 e^{\alpha t}L(\gamma_1 (t), \dot \gamma_1 (t))dt ,$
 hence
\begin{align*}
  u_{H, \alpha}(Q) - &u_{H, \alpha}(q) \leq \int_{-1/A}^0e^{\alpha \sigma}\big( L(\gamma_1(\sigma), \dot\gamma_1(\sigma))-L(\gamma(\sigma), \dot\gamma(\sigma))\big)d\sigma.\\
     &\leq   \int_{-1/A}^0e^{\alpha \sigma}\Big(  \partial_qL(\gamma(\sigma), \dot\gamma(\sigma))(\gamma_1(\sigma)-\gamma(\sigma))+\partial_vL(\gamma(\sigma), \dot\gamma(\sigma))(\dot\gamma_1(\sigma)-\dot\gamma(\sigma))\\
  &\quad\quad\quad\quad+C\| \gamma_1(\sigma)-\gamma(\sigma)\|^2+C\| \dot\gamma_1(\sigma)-\dot\gamma(\sigma)\|^2\Big)d\sigma   
\end{align*}
for some constant $C>0$.

  Using integration by parts, and the fact that $\gamma$ is a solution of the Euler-Lagrange equations,  we deduce
    \begin{align*}  u_{H, \alpha}(Q)- u_{H, \alpha}(q)
\leq \Big[ e^{\alpha\sigma}\partial_vL(\gamma(\sigma), \dot\gamma(\sigma))(\gamma_1(\sigma)-\gamma(\sigma))\Big]_{-1/A}^0+C'\| Q-q\|^2
    \end{align*}
 for some $C'>0$ independent of $A$. Letting $A$ go to infinity, this yields    $$u_{H, \alpha}(Q)\leq  u_{H, \alpha}(q)+\partial_vL(\gamma(0), \dot\gamma(0))(Q-q)+C'\| Q-q\|^2.$$
 This shows that $\partial_vL(\gamma(0), \dot\gamma(0))$ is a super-differential of $u_{H, \alpha}$.
  \end{proof}
 
  \begin{prop}\label{prop:u-and-U}
  For every $q\in N$, we have
  $$u_{H, \alpha}(q)=\min_{p\in T_q^*N\cap K_{H, \alpha}}U_{H, \alpha}(q,p).$$
  Moreover, at every $p\in T_q^*N\cap K_{H, \alpha}$ where the minimum is attained, $p$ is a super-differential of $u_{H, \alpha}$ at $q$.
  \end{prop}
  \begin{proof} For $x\in K_{H, \alpha}$, we set $(q_t, p_t)=\phi_{-H, \alpha}^t(x)$. Then 
  $$\lambda_{(q_t, p_t)}(X_{-H, \alpha}(q_t, p_t))-H(q_t, p_t)=p_t\dot q_t-H(q_t, p_t)=L(q_t, \dot q_t).$$
  We deduce from \eqref{eq:expression-DLO-fixedpoint } that
  $$U_{H, \alpha}(x)=\int_{-\infty}^0 e^{\alpha t}L(q_t, \dot q_t)dt\geq u_{H, \alpha}(q).$$
  We know that there is a solution $\gamma:]-\infty, 0]\to N$ of the Euler-Lagrange equations   for which we have equality in \eqref{eq:expression-DLO-fixedpoint } for $q=q_0$. This implies that $x=(\gamma(0), \partial_vL(\gamma(0), \dot \gamma(0))$ satisfies
  $$U_{H, \alpha}(x)=\int_{-\infty}^0e^{\alpha t}L(\gamma(t), \dot\gamma (t))dt=u_{H, \alpha}(q_0).$$
  \end{proof}
  
   \subsection{Graph selectors}\label{sec:graph-selectors}The first version of graph selector is due to Sikorav and Chaperon (\cite{Sik,Chaperon-HJ}), and was studied in 
   \cite{Viterbo-Ottolenghi,Viterbo-Ott} in the case of a Lagrangian Hamiltonianly isotopic to the zero section (this is always the case for the Lagrangian occurring in the evolution Hamilton-Jacobi equation) and using Floer homology by Oh (\cite{Oh-action-functional-I}). The general case (without assuming $L$ Hamiltonianly isotopic to $0_N)$ was  first written by  Amorim-Oh-Santos (\cite{Amorim-Oh-Santos}). In a different vein, the selectors are defined by Guillermou from the sheaf-theoretic viewpoint (\cite{Vichery-these,Guillermou-Asterisque}).
    
   We recall that every $\tilde L=(L, f_L)\in\LL(T^*N,\omega)$ has a unique graph selector that is a Lipschitz continuous function $u_{\tilde L}:M\to \R$. If $L$ is $C^k$, there exists an open subset $U_{\tilde L}$ of $N$ with full Lebesgue measure on which $u_{\tilde L}$ is $C^{k+1}$ and such that for any $q\in U_{\tilde L}$ we have:
   $$du_{\tilde L}(q)\in L\quad\text{and}\quad u_{\tilde L}(q)=f_L(du_{\tilde L}(q)).$$
   Moreover, if $\tilde L_i=(L_i, f_{L_i})\in\LL(T^*N,\omega)$ for $i=1, 2$, then 
   $$\| u_{\tilde L_1}-u_{\tilde L_2}\|_\infty\leq c(\tilde L_1, \tilde L_2).$$
   Hence the graph selector can be continuously extended to the completion $\hatLL(T^*N,\omega)$.
   
      The value of $u_{\tilde{L}}(x)$ is in fact a spectral invariant for the pair $(L,T_x^*N)$ and the above inequality follows from the triangle inequality in Proposition \ref{prop:lag-spec-inv-properties} applied to the Lagrangians $L_1$, $L_2$, $T_x^*N$ (strictly speaking Proposition \ref{prop:lag-spec-inv-properties} does not apply directly since $T_x^*N$ is not compact, but its proof extends verbatim to this setting, see \cite[Theorem 16]{Hum-Lec-Sey3}). 
      When $L$ is Hamiltonian isotopic to $0_N$, the above inequality also follows from the reduction inequality from \cite{Viterbo-STAGGF} and is mentioned for example in \cite{Cardin-Viterbo}, p.263. The general case can also be proved for example from \cite{Vichery-these}, Prop. 8.13 by taking  
   $\mathcal F_2=k_x, {\mathcal F}_1=\mathcal F_{L_1}, {\mathcal F}_3={\mathcal F}_{L_3}$ yielding $u_{L_1}(x)-u_{L_2}(x)\geq c(1;L_1,L_2)$.

   \begin{prop}
   For every $u\in C^2(N, \R)$, $T_{H, \alpha}^tu$ is the graph selector of $\Phi_{-H, \alpha}^t(\text{graph}(du), u)$.
   \end{prop}
   \begin{proof}
It is proven in \cite{Roos2019} that when $u$ is $C^2$ and $H_t$ is Tonelli and time dependent, then $T^t_{H, 0}u=u_{\Phi^t_{-H, 0}(\text{graph}(du), u)}$. In our case, we consider the isotopy
   $(\phi^{-t}_{0, \alpha}\circ \phi_{-H, \alpha}^{t})$. It is generated by the vector field $$Y=(\phi_{0,\alpha}^{-t})_* X_{-H,\alpha}-X_{0,\alpha},$$
   which is associated to the Hamiltonian $-\Hh$, where
   $$\Hh_t=(e^{-\alpha t}H\circ \phi_{0, \alpha}^{t}).$$
   As $\phi_{0, \alpha}^{t}$ preserves the fibers and its restriction to every fiber is linear, $\Hh_t$ is Tonelli.  Moreover, since $\phi_{0,\alpha}^t(q,p)=(q,e^{-\alpha t}p)$, the Lagrangian function associated to $H_t$ is $e^{-\alpha t}L$. We deduce that the Lax-Oleinik semi-group $(T^t)$ associated to $\Hh$ is related to the discounted Lax-Oleinik semi-group $(T^t_{H, \alpha})$ by
   $$T_{H, \alpha}^tu=T^t(e^{-\alpha t}u).$$
   Because $\phi_{-H, \alpha}^t =\phi^t_{-\Hh}\circ \phi_{0, \alpha}^t$, we have also $\Phi_{-H, \alpha}^t =\Phi^t_{-\Hh}\circ \Phi_{0, \alpha}^t$ and then for every $u\in C^2(N, \R)$ 
   $$\Phi_{-H, \alpha}^t(\text{graph}(du), u)=\Phi_{-\Hh}^t(\text{graph}(e^{-\alpha t}du), e^{-\alpha t}u).$$
   Finally, as $\Hh$ is Tonelli, we know that $T^t(e^{-\alpha t}u)$ is the graph selector of the Lagrangian $\Phi_{-\Hh}^t(\text{graph}(e^{-\alpha t}du), e^{-\alpha t}u)$ and this gives the wanted result.
   \end{proof}
   \begin{cor}\label{cor:KAMF-graph-selector}
   $u_{H, \alpha}$ is the graph selector of $\tilde L_\infty(H, \alpha)$.
 \end{cor}

Note that, as shown in Appendix \ref{appendix}, this corollary does not hold without the Tonelli assumption on the Hamiltonian.

  \begin{proof}
  We pick any $u\in C^2(N,\R)$, for instance $u=0$. We know that $\Phi_{-H,\alpha}^t(\text{graph}(du), u)$ $c$-converges to $\tilde L_\infty(H, \alpha)$ as $t$ goes to $+\infty$. We deduce that the graph selector $T^t_{H,\alpha}u$ of $\Phi_{-H,\alpha}^t(\text{graph}(du), u)$  converges uniformly to the graph selector of $\tilde L_\infty(H, \alpha)$ as $t$ goes to $+\infty$. But we also know that $T^t_{H,\alpha}u$ converges to $u_{H, \alpha}$.    This concludes the proof.
  \end{proof}
  
    \subsection{Proof of Theorem \ref{th:discounted} }\label{sec:proof-discounted}

The following lemma records formulas that will be useful later in the proof.

    \begin{lem}\label{lem:conjugacy-Hamiltonian-exact-symplectic}
      Let $\psi$ be an exact symplectic diffeomorphism of $T^*N$ and $S$ be a function satisfying $\psi^*\lambda-\lambda=dS$. We  consider $F=H\circ \psi -\alpha S$. Then the following identities hold:
      \begin{align*}
      &X_{-F,\alpha }=\psi^*X_{-H,\alpha}, \quad \phi^t_{-F,\alpha}=\psi^{-1}\circ \phi_{-H, \alpha}^t\circ \psi,\quad K_{F,\alpha}=\psi^{-1}(K_{H,\alpha}),\\  & L_\infty(H, \alpha)=\psi(L_\infty(F, \alpha)),\quad U_{F, \alpha}=U_{H,\alpha}\circ \psi-S.
      \end{align*}
    \end{lem}
    
    \begin{proof} The vector field associated to the flow $(\psi^{-1}\circ \phi^t_{-H, \alpha}\circ \psi)$ is  $Y=\psi^*X_{-H,\alpha}$, thus the second identity follows from the first one. Moreover, since $\psi$ is symplectic, we have 
\begin{equation*}\iota_Y\omega=\psi^*(\iota_{X_{-H,\alpha}}\omega)=\psi^*dH+\alpha\psi^*\lambda=d(H\circ\psi)+\alpha\lambda +\alpha dS=dF+\alpha\lambda.
\end{equation*}
Hence $Y=X_{-F, \alpha}$. 
We deduce that $K_{F,\alpha}=\psi^{-1}(K_{H,\alpha})$ and $X_{-F,\alpha}=\psi^*X_{-H,\alpha }$.  We deduce from Proposition \ref{prop:gamma-support} that $L_\infty(H, \alpha)=\psi(L_\infty(F, \alpha))$.

Finally, for $x\in K_{F,\alpha}$, we have
\begin{align*}
U&{}_{F, \alpha}(x)=\int_{-\infty}^0e^{\alpha t}\big( \lambda_{\phi^t_{-F,\alpha}(x)}X_{-F,\alpha} (\phi_{-F,\alpha}^t(x))-F(\phi^t_{-F,\alpha}(x))\big)dt\\
&=\int_{-\infty}^0e^{\alpha t}\Big( \lambda_{\psi^{-1}\phi^t_{-H,\alpha} \psi(x)}\big((D\psi(x))^{-1}X_{-H,\alpha}(\phi^t_{-H, \alpha} \psi(x))\big)-F(\phi^t_{-H,\alpha}(x))\Big)dt\\
&=\int_{-\infty}^0e^{\alpha t}\Big(\big(  \lambda_{\phi^t_{-H,\alpha}(\psi(x))}\big(X_{-H,\alpha} (\phi^t_{-H,\alpha}\psi(x))\big)-d(S\circ \psi^{-1})(\phi^t_{-H, \alpha} \psi(x))X_{-H,\alpha} (\phi^t_{-H,\alpha}\psi(x))\big)\\
&\quad\quad -\big(H(\phi^t_{-H, \alpha}\psi(x))-\alpha S\circ\psi^{-1}(\phi^t_{-H,\alpha} \psi(x))\big)\Big)dt\\
&=U_{H, \alpha}\psi(x))-\int_{-\infty}^0\frac{d}{dt}\big( e^{\alpha t}S\circ\psi^{-1}(\phi^t_{-H,\alpha}\psi(x))\big)dt\\
&=U_{H,\alpha}(\psi(x))-S(x).
\end{align*}
\end{proof}

\begin{proof}[Proof of Theorem \ref{th:discounted}]. Let $q\in N$ be a point where $u_{H, \alpha}$ has a derivative. Let  $B$ be a small ball centered at $(q, du_{H, \alpha}(q))$. Our goal is to build a Hamiltonian diffeomorphism $\psi$ with support in $B$ so that $L_\infty(H, \alpha)\neq\psi(L_\infty(H, \alpha))$. By definition of the $\gamma$-support, this will imply that $(q, du_{H, \alpha}(q))\in B_{H, \alpha}$ and conclude our proof of Theorem \ref{th:discounted}.

Let $(f_t)$ be a $C^2$-small isotopy of diffeomorphisms of the cotangent fiber $T^*_qN$ so that $f_0=\text{Id}_{T^*_qN}$, the support of the isotopy is in a small ball centered at $du_{H, \alpha}(q)$  and we have $f_1(du_{H, \alpha}(q))\neq du_{H,\alpha}(q)$.
    
    We extend $(f_t)$ to a symplectic isotopy $(g_t)$ in a small Darboux chart by the formula 
    $$g_t(Q, P)=(q+{}^tDf_t(P)^{-1}(Q-q), f_t(P)).$$
    Since $(f_t)$ is $C^2$-close to identity, $g_t$ is $C^1$-close to identity, thus its admits a generating function $\tau_t(Q, P_t)$ such that
 \begin{equation} \label{eq:generating-function }  g_t(Q, P)=(Q_t, P_t)\Longleftrightarrow
   \begin{cases}
Q_t=\partial_{P_t}\tau_t(Q, P_t)\\ P=\partial_Q\tau_t(Q, P_t)
\end{cases}
\end{equation}
    More precisely, we have $$\tau_t(Q, P_t)=\langle (f_t)^{-1}(P_t), Q-q\rangle+\langle P_t,q\rangle.$$
    Using a bump  function $\eta$ around $q$, we set
    $$\sigma_t(Q, P)=\eta(Q)\tau_t(Q, P)+(1-\eta (Q))\langle P,Q\rangle.$$
    The function  $\sigma_t$ is the generating function in the sense of \eqref{eq:generating-function }  of an exact  symplectic diffeomorphism $h_t$, which  is close to identity, and whose support is contained in a small neighbourhood of $(q, du_{H, \alpha}(q))$. As $\sigma_t=\tau_t$ in a neighbourhood of $T^*_qN$, we have $h_t(T^*_qN)=T^*_qN$, $h_1(q, du_{H, \alpha}(q))\neq (q, du_{H, \alpha}(q))$ and $\sigma_{t|T^*_qN}=\tau_{t|T^*_qN}$.
    
    Moreover, if we let $(Q_t, P_t)$ denote $h_t(Q, P)$, we have
    $$h_t^*\lambda-\lambda=P_tdQ_t-PdQ=d(\langle P_t,Q_t\rangle)-Q_tdP_t-PdQ=d(\langle P_t,Q_t\rangle)-d\sigma_t.$$
    Thus, we have $h_t^*\lambda-\lambda=d\Sigma_t$ where $\Sigma_t=\langle P_t,Q_t\rangle-\sigma_t$.
    Note that on the fiber $T_q^*N$, i.e. for $Q=q$, we have 
    $$\Sigma_t
    =\langle P_t,Q_t\rangle-\tau_t=\langle Q_t-q, P_t\rangle-\langle (f_t)^{-1}(P_t), Q-q\rangle=0.$$ 
    
    As $h_1$ is close to identity and $\Sigma_1$ close to the zero function, the Hamiltonian $F=H\circ h_1-\alpha \Sigma_1$ is Tonelli, which implies by Corollary \ref{cor:KAMF-graph-selector} that $u_{F,\alpha}$ is the graph selector of $\tilde L_\infty(F, \alpha)$. By Lemma \ref{lem:conjugacy-Hamiltonian-exact-symplectic}, we also know that $L_\infty(F, \alpha)=h_1^{-1}(L_\infty(H, \alpha))$. Thus, if we prove that $u_{F, \alpha}-u_{H, \alpha}$ is not a constant function, we will deduce that $L_\infty(H, \alpha)\neq L_\infty(F, \alpha)$, hence $L_\infty(H, \alpha)\neq h_1^{-1}(L_\infty(H, \alpha))$. We will then have reached our goal explained at the beginning of the proof. 
    
    By Lemma \ref{lem:conjugacy-Hamiltonian-exact-symplectic}, $K_{F, \alpha}=h_1^{-1}(K_{H, \alpha})$, hence 
    $$K_{F, \alpha}\cap T_q^*N=h_1^{-1}(K_{H, \alpha}\cap T^*_qN).$$
    We also have $U_{F, \alpha} = U_{H, \alpha}\circ h_1-\Sigma_1$  and $\Sigma_1|_{T^*_qN}=0$, hence $U_{F, \alpha}|_{T^*_qM}=U_{H, \alpha}\circ h_{1}|_{T_q^*N}$. Using Proposition \ref{prop:u-and-U}, we deduce the following equalities:
    \[\begin{split}
u_{H, \alpha}(q)=\min_{p\in T^*_qN\cap K_{H, \alpha}}U_{H, \alpha}(q,p)=U_{H, \alpha}(q, du_{H, \alpha}(q))\\
\min_{p\in T^*_qN\cap K_{H, \alpha}}U_{F, \alpha}\circ h_1^{-1}(q,p)=U_{F, \alpha}\circ h_1^{-1}(q, du_{H, \alpha}(q))\\
\min_{p\in T^*_qN\cap K_{F, \alpha}}U_{F, \alpha}(q,p)=U_{F, \alpha}\big( h_1^{-1}(q, du_{H, \alpha}(q))\big)\\
u_{F, \alpha}(q)=U_{F, \alpha}\big( h_1^{-1}(q, du_{H, \alpha}(q))\big).
\end{split}\]
Proposition \ref{prop:u-and-U} then implies that $h_1^{-1}(q, du_{H, \alpha}(q))$ is a super-differential of $u_{F, \alpha}$ at $q$. Since by construction $h_1^{-1}(q, du_{H, \alpha}(q))\neq (q, du_{H, \alpha}(q))$, we deduce that $u_{F, \alpha}-u_{H, \alpha}$ admits a non-zero super-differential, hence is not a constant function.
\end{proof}

        \subsection{The time dependent case }\label{sec:time-dependent}
        We can adapt the above proofs to the time-dependent setting. We assume that  $H:T^*N\times {\mathbb T}\to \R$ is a Tonelli Hamiltonian such that:
        \begin{itemize}
        \item the time-dependent vector field $X_{-H_t, \alpha}$ is complete,
        \item there is a compact neighborhood of the zero section which is forward invariant. 
        \end{itemize}
 The evolutive discounted Hamilton-Jacobi equation is in this case
        \begin{equation}
        \label{eq:HJTD-discounted}
\tfrac d{dt}u_t(x)+\alpha u_t(x) + H_t(x,du_t(x))=c.
\end{equation}
The Lagrangian action functional $\mathcal A_L$ is defined for $t_1<t_2$ and $x, y\in N$ by $$\mathcal  A_L(x, t_1;  y; t_2)=\inf_{\substack{\gamma:[t_1, t_2]\to N\\t_1\mapsto x, t_2\mapsto y}} \int_{t_1}^{t_2}e^{\alpha s}L(s;\gamma(s), \dot \gamma (s))ds.$$
where $L$ is the Lagrangian function associated to $H$. 
The cost function $c:N\times N\to \R$ is defined by $c(x,y)=\mathcal A_L(x, -1; y, 0)$.  Observe that it is a continuous function. 
For every $u\in C^0(N, \R)$ and $t_1>t_2$, we define
$$T_{H, \alpha}^{t_2, t_1}u(q)=\inf_{q'\in N} \big(e^{\alpha (t_2-t_1)}u(q')+e^{-\alpha t_1}\mathcal A_L (q', t_2; q, t_1)\big).$$
We have then $T_{H, \alpha}^{t_2, t_1}\circ T_{H, \alpha}^{t_3, t_2}=T_{H, \alpha}^{t_3, t_1}$ and $T_{H, \alpha}^{t_1+1, t_2+1}=T_{H, \alpha}^{t_1, t_2}$. The discrete Lax-Oleinik operator is also defined by$\mathcal T=T_{H, \alpha}^{0, 1}$, i.e.,
$$\mathcal T u(q)= \inf_{q'\in N} \big( e^{-\alpha}u(q')+c(q', q)\big).$$
The weak KAM solution is then the unique fixed point of $\mathcal T$.  Using the continuous-time dependent setting, we can adapt the proof of the autonomous case and deduce that         
 the pseudo-graph of the weak KAM solution is contained in 
 Birkhoff attractor      of the time-one map of the Hamiltonian isotopy.

\section{The limit \texorpdfstring{$\alpha \to 0$}{a to 0} and the pendulum with friction}\label{sec:alpha-to-1}

The goal of this section is twofold. First we want to understand what happens to the Birkhoff attractor of the composition of a conformal and a
Hamiltonian diffeomorphism when the conformal factor $a=e^{-\alpha}$ converges to $1$. We shall see that there are several possible limits, yielding invariant sets. One could hope that this 
limit of invariant sets, each of which is a $\gamma$-support of some element $L_\infty(\alpha)$ in $\widehat{\LL} (M,\omega)$ could correspond to the $\gamma$-limit of the $L_\infty(\alpha)$.  In a second part we show that this is unfortunately not the case even in the simple case of the pendulum.  
\subsection{Invariant sets for Hamiltonian flows}
Let us consider a Liouville vector field, with flow $\chi^\alpha$, so that $(\chi^\alpha)^*\lambda=e^{-\alpha} \lambda$. Let $\phi$ be a Hamiltonian diffeomorphism. Then  $\chi^\alpha\circ \phi$ is also conformal of ratio $e^{-\alpha}=a$. For each $\alpha>0$, we let $B^\alpha(\phi)$ denote the Birkhoff attractor $B(\chi^\alpha\circ \phi)$. Using compactness of the Hausdorff topology on compact sets we may define

\begin{defn}  We assume that the $B^\alpha(\phi)$ remain in a compact subset as $\alpha$ goes to $0$. We denote by $B^-(\phi)$ (resp. $B^+(\phi)$) the inferior limit (resp. superior limit) of $B_{}^\alpha(\phi)$, i.e.
$$B^-(\phi)=\liminf_{\alpha\to 0}B^\alpha(\phi)$$
$$B^+(\phi)=\limsup_{\alpha\to 0}B^\alpha(\phi)$$
Moreover, 
we denote by $B^0(\phi)$
 any $\limsup$ or $\liminf$ of $B^{\alpha_k}(\phi)$ for some sequence $\alpha_k$ going to $0$. 
\end{defn} 

\begin{rem}
  Such limits exist by compactness of the Hausdorff distance on subsets of a compact set. If $B^-(\phi)=B^+(\phi)$, then 
  we just have one Hausdorff limit and $B^0(\phi)=B^-(\phi)=B^+(\phi) $ (see \cite[p.26, Exercice 4.23, 4.24]{Kechris}).
\end{rem}

 \begin{prop} Any set $B^0(\phi)$ is invariant by $\phi$. The same holds for $B^-(\phi)$.
\end{prop}

\begin{proof} 
Let $x\in B_{}^0(\phi)=\limsup_{k\to\infty}B^{\alpha_k}(\phi)$.
Then there exists a subsequence $(\beta_k)$ of $(\alpha_k)$, such that  $x=\lim_kx_k$ where $x_k\in B_{}^{\beta_k}(\phi)$. If $B_{}^0(\phi)=\liminf_{k\to\infty}B^{\alpha_k}(\phi)$, the same holds with $\beta_k=\alpha_k$.

By assumption $\phi(x_k)\in \chi^{-\beta_k}B^{\beta_k}(\phi)$, and  $d(B^{\beta_k}(\phi), \chi^{-\beta_k}B^{\beta_k}(\phi))$ goes to $0$ as $\beta_k$ goes to $0$, so $\phi(x_k)$ converges to $\limsup \chi^{-\beta_k}B^{\beta_k}(\phi))=B^0(\phi)$. This means that $\phi(x)$ is in $B^0(\phi)$. 
\end{proof}

\begin{rems} 
	\begin{enumerate} 
 \item If we have a common bound for all the $B^\alpha(\phi)$ (for example if $H$ is autonomous and Tonelli), we have at least one non-empty invariant set. Of course it could be the whole space, but this can often be excluded, for example if for $c$ sufficiently large regular value,  $\{H=c\}$ is transverse to the Liouville vector field.
 \item The subset $B^+(\phi)$ is always non empty. A priori $B^-(\phi)$ could be empty, but by Proposition \ref{prop:weakKAM-non-discounted} below this cannot happen in the Tonelli case.  
            \item  If $(X_k)_{k\geq 1}$ is a family of subsets in a metric space,  $x\in \liminf_k X_k$ if and only if $\lim_k d(x,X_k)=0$,
 	while $x\in \limsup_k X_k$ if and only if the closure of the sequence $(d(x,X_k))_{k\geq 1}$ contains $0$. 
 	Note that if the $X_k$ are contained in a compact set, $X_k$ converges for the Hausdorff distance to $X_\infty$ if and only if $X_\infty=\limsup_k X_k=\liminf_k X_k$ (see \cite{Kechris}, pp. 25-26).
  \end{enumerate} 
  \end{rems} 

  \begin{prop}\label{prop:weakKAM-non-discounted} If $H$ is Tonelli and autonomous, then $B^-(\phi_{-H})$ contains the graph of the weak KAM solution $u_0$ of the Hamilton-Jacobi equation $H(x,d_xu)=0$ which, by \cite{DFIZ}, is the limit as $\alpha$ goes to zero of the functions $u_{H,\alpha}$.    
  \end{prop}

  \begin{proof} By Theorem \ref{th:discounted}, $\mathrm{graph}(du_{H,\alpha})$ is included in $B^\alpha(\phi_{-H})$  and in $K(H,\alpha)$, see \eqref{EKH}. If we  look at the proof of Proposition \ref{prop:super-differential},  we see that the constant $C$ of semi-concavity that appears in the proof is a little larger than the maximum of the $C^2$-norm of $L$ 
    at $(\gamma(0),\dot\gamma(0))$ where $\gamma :]-\infty, 0]\to M$ is a curve where the minimum is attained in Equation \eqref{eq:expression-DLO-fixedpoint }. Moreover,  $\partial_vL(\gamma(0), \dot\gamma(0))$ is 
      a super-differential of $u_{H, \alpha}$ at $q=\gamma(0)$ and thus is contained in the union of the convex hull of $\overline{\text{graph}(du_{H, \alpha})}\cap T_q^*M$, see Proposition 3.3.4 of \cite{CanSin04}.
      
We have proved that for every $\alpha>0$, there are inclusions $\text{graph}(du_{H, \alpha})\subset B^{\alpha}(\phi_{-H})\subset K(H,\alpha)$. We deduce that $\partial_vL(\gamma(0), \dot\gamma(0))$ is contained in the compact set $\mathcal C_{H}$ that is obtained by taking the fiberwise convex hull of $K(H,\alpha)$. Hence the semi-concavity constant of every $u_{H, \alpha}$ is less than $\sup\{ \| p\|; (q, p)\in N\}$ where $N$ is a fixed compact neighbourhood of $\mathcal C_H$.
  
Therefore, the functions $u_{H,\alpha}$ are uniformly semi-concave and uniformly converge to $u_0$, thus according to \cite{Attouch} (who proved it in the convex case, but this immediately implies the uniformly semi-concave case)
    \[\mathrm{graph}(du_0)\subset \liminf_{\alpha\to 0}\,\mathrm{graph}(du_{H,\alpha}).\] 
The proposition follows.
  \end{proof}

\begin{rem} 
  There are many possible choices for $\lambda$ and hence $\chi^t$. For example in $T^*N$ we can  replace the tautological 1-form $\lambda$ by $\lambda-\pi^*\mu$ where $\mu$ is a closed 1-form on $N$ and $\pi:T^*N\to N$ is the canonical projection. Then  let $\chi_0^\alpha(q,p)=(q,e^{-\alpha}p)$ and $\chi_\mu^\alpha(q,p)=(q,\mu(q)+e^{-\alpha}(p-\mu(q))$.  Applying this to the above, we get for any closed 1-form $\mu\in H^1(N, \mathbb R)$ invariant subsets $B_{\mu}^0(\phi), B_{\mu}^-(\phi), B_{\mu}^+(\phi)$.
  Note that these subsets depend on $\mu$ and not only on its cohomology class. For example, for $\mu=df$ we have $B^\alpha(\Id)=\gra(df)$
\end{rem}

\begin{defn}  Let $(\phi_t)$ be an isotopy generated by a 1-periodic Hamiltonian such that $\phi_0={\rm Id}$ and $\phi_1=\phi$. For every $x\in T^*N$, consider the linear maps $g_{k, x}:Z^1(N)\to \R$ defined by \[\eta\mapsto\frac{1}{k}\int_0^k\eta(\dot \phi_t(x))dt.\]
 Every limit point $g_\infty$ of $(g_{k, x})_{k\in\N}$ is also linear and vanishes on the exact forms, hence defines an element $g_\infty\in H_1(N)$  which we call a \emph{rotation vector}.  The set of all rotation vectors for $x \in B_{\mu}^*(\phi)$ is denoted by $R^*(\mu,\phi)$. 
\end{defn} 
This can be used to prove that we get distinct sets $B_{\mu}^0(\phi)$. Indeed, whenever $R^0(\mu,\phi)\neq R^0(\mu',\phi)$ we have $B_{\mu}^0(\phi)\neq B_{\mu'}^0(\phi)$ and whenever 
$R(\mu,\phi)\cap R(\mu',\phi)=\emptyset$ we have $B_{\mu}^0(\phi)\cap B_{\mu'}^0(\phi)=\emptyset$. 

\subsection{The pendulum with friction}
One could ask, since the invariant sets $B^\alpha_{\mu}(\phi)$ are the $\gamma$-supports of some $L^\alpha_{\mu,\infty}(\phi) \in \widehat{\mathcal L} (T^*N)$, is it true that $\gamma-\lim_{\alpha\to 0} L^\alpha_{\mu,\infty}(\phi)=L^0_{\mu,\infty}(\phi)$ with $\gammasupp (L^0_{\mu,\infty}(\phi))=\B^0_{\mu}(\phi)$. As we shall see, this is not the case in general. 

\begin{prop} 
Let $H(\theta,p)= \frac{1}{2}p^2-\cos(\theta)$ be the Hamiltonian for the pendulum on $T^*\bS^1$ and set $\mu=0$. Then  the sequence $L^\alpha_{0,\infty}(\phi_{-H})$ has no limit point as $ \alpha$ goes to $0$.
\end{prop}
 
Let us consider the equation of the pendulum with friction $\alpha$, that is for $(\theta, p)\in T^*\bS^1$
\begin{equation}\label{eq:pendulum} \ddot \theta + \alpha \dot \theta + f(\theta)=0
 \end{equation} 
We shall assume $f(\theta)=\sin(\theta)$, but the same results would hold for any $f$ such that
\begin{enumerate} 
	\item  $f(\theta)=0 \Leftrightarrow \theta \in \{0,\pi\}$
\item  $f'(0)>0$  and $f'(\pi)<0$ 
\end{enumerate}
We write the equation as
\begin{equation}\label{eq:pendulum2}\left\{ \begin{array}{ll} \dot \theta =p\\
\dot p=- \alpha  p - f(\theta) \end{array}\right . 
 \end{equation} 

Note that for $\alpha=0$ we have the standard pendulum equation, and the vector field corresponding to the equation is 
$(p, -\alpha p - f(\theta))=(p,-f(\theta)) - \alpha (0,p)$. In our conventions, the vector field $(0,-p)$ is the Liouville vector field for $p\,d\theta$, while $(p, -f(\theta))$ is the Hamiltonian vector field corresponding to $-H(\theta, p)= - \frac{1}{2}p^2+ F(\theta)$ where $F'(\theta)=f(\theta)$. In other words, these equations generates the flow $\phi_{-H,\alpha}^t$. We shall always assume $\alpha >0$.

Note that the equilibrium points are given by $p=0, f(\theta)=0$, so there are only two equilibria, one at $(0,0)$, since $f'(0)>0$ it is a stable focus, and one at $\theta=\pi$ with $f'(\pi)<0$, a saddle. 
Finally note that the time-one flow is a conformally symplectic map with ratio $a=e^{-\alpha}$.

\begin{prop} 
The origin $(0,0)$ is a stable equilibrium, while $(0,\pi)$ is unstable. There is a single pair of heteroclinic orbits, from $(\pi,0)$ to $(0,0)$, that we denote $\gamma_L$ and $\gamma_R$ defined on $ \mathbb R$, such that $\lim_{t\to +\infty}\gamma_{\alpha,R}(t)=\lim_{t\to +\infty}\gamma_{\alpha,L}(t)=(0,0)$ while $\lim_{t\to -\infty}\gamma_{\alpha,R}(t)= \lim_{t\to -\infty}\gamma_{\alpha,L}(t)=(\pi,0)$. 
The Birkhoff attractor is $B_\alpha=\gamma_{\alpha,R}( \mathbb R) \cup \gamma_{\alpha,L}( \mathbb R)$.
\end{prop}

\begin{proof} 
We refer to \cite{Martins-Birkhoff}, where it is proved that the largest bounded invariant set is $B_\alpha$. Since there is no smaller non-trivial continuum that is an invariant set, $B_\alpha$ must be the Birkhoff attractor.
\end{proof}

\begin{figure}[ht]
 \begin{overpic}[width=9cm]{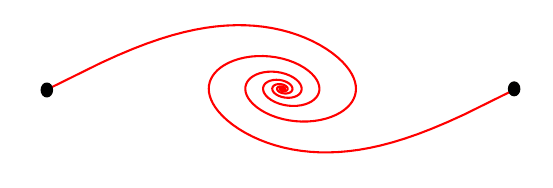}
 \end{overpic}
\caption{The Birkhoff attractor for the pendulum with friction}
\label{fig-3}
\end{figure}

\begin{figure}[ht]
 \begin{overpic}[width=9cm]{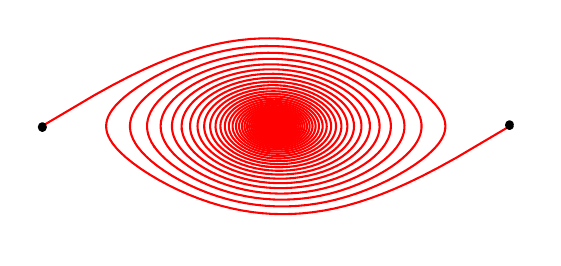}
 \end{overpic}
\caption{ The Birkhoff attractor for the pendulum with very small friction}
\label{fig-4}
\end{figure}

\begin{figure}[ht]
 \begin{overpic}[width=9cm]{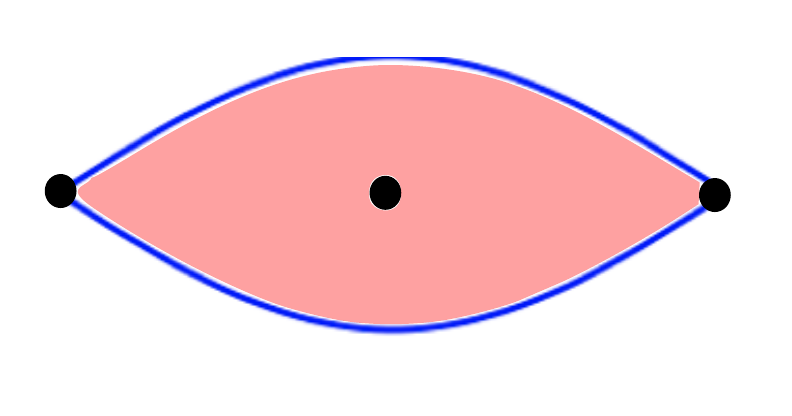}
 \end{overpic}
\caption{ The limit of the Birkhoff attractors as $\alpha$ goes to $0$ for the pendulum}
\label{fig-16}
\end{figure}
Note that in our case, the Birkhoff attractor is a $C^0$-curve which is the image of the zero section by a Hamiltonian homeomorphism. We claim that as such, it is the $\gamma$-support of a unique element of $\hatL$ denoted by $L_\alpha$. To explain this point, we call $(U)$ the property satisfied by a subset of $T^*{\mathbb S}^1$ if and only if it is the $\gamma$-support of a unique element $L\in\hatL$. It follows from \cite[Theorem 8.6]{Viterbo-gammas} that the zero section satisfies $(U)$. Moreover, by Remark \ref{rem:sympeo}, the property $(U)$ is invariant under the action of Hamiltonian homeomorphisms on $\hatL$. This shows our claim that $B_\alpha$ satisfies $(U)$.

In particular, $L_\alpha$ is the fixed point provided by Theorem \ref{th:main-for-branes}.
  We now wonder whether $L_\alpha$ converges or at least has a converging subsequence as  $\alpha$ goes to $0$. 
  The following proposition answers this question by the negative.
\begin{prop} 
Let $(\alpha_k)_{k\geq 1}$ be a sequence of positive real numbers converging to $0$. Then  the sequence  $(L_{\alpha_k})_{k\geq 1}$ does not $\gamma$-converge. 
\end{prop} 
This will immediately follow from
\begin{prop}\label{Prop-6.9} For any $\alpha>0$ we can find $0<\beta_0<\alpha$ so that for $0<\beta < \beta_0$ we have 
$$\gamma (L_\alpha, L_\beta) \geq 4 \left(1 - \frac{\beta}{\alpha}\right) $$
 \end{prop} 

   Note that the curves $L_\alpha$ are not smooth as --- at least for $\alpha$ small enough\footnote{in fact $\alpha<2$, which implies the equilibrium is elliptic.} ---  they twist infinitely many times around the point $(0,0)$. We will need to approximate them by smooth curves which are spirals described as follows. In the phase space $\R/2\pi\Z\times\R$ of the pendulum, we consider the closed disc $D$ of radius $\pi$ centered at $(0,0)$ and endow it with polar coordinates $(r,\phi)$. For any $s>0$, we consider homeomorphisms $\rho$ supported in $D$ and of the form given by
\[\rho(r,\phi)=(r,\phi+h(r)), \quad \forall (r,\phi)\in D\]
   where  $h:[0,\pi]\to[-s,0]$ is an increasing continuous map with $h(0)=-s$ and $h(\pi)=0$. We will say that a curve $L_2$ is a \emph{smooth spiral} if it is an embedded smooth closed curve in $\R/2\pi\Z\times\R$ transverse to the zero section $L_1= \R/2\pi\Z\times\{0\}$ and if there exist a parameter $t>0$ and an orientation preserving  homeomorphism $\psi$ of $\R/2\pi\Z\times\R$ which fixes $L_1$ and satisfies $L_2=\psi(\rho^t(L_1))$ (See Figure \ref{fig-10}). The above definition also makes sense if $s=\infty$, in this case, the curve $L_2$ is only continuous and we call it an \emph{infinite spiral}. The argument given above shows that infinite spirals satisfy Property $(U)$, i.e. are the $\gamma$-support of a unique element in $\hatL$. In the argument below, we sometimes abuse notation and also denote by $L_2$ this unique element.

\begin{figure}[ht]
 \begin{overpic}[width=9cm]{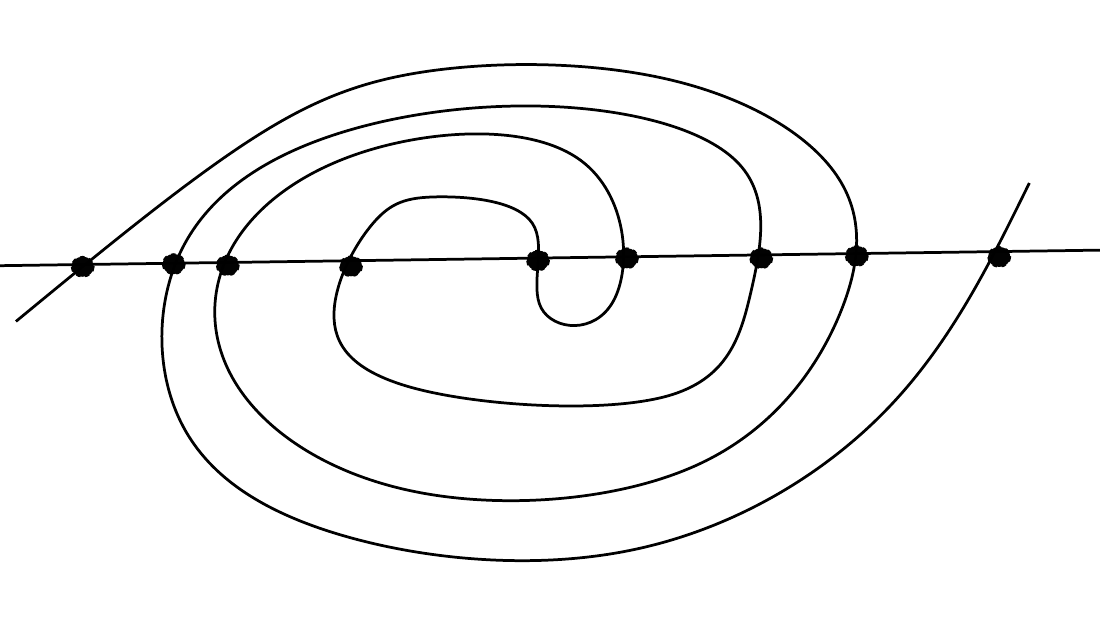}
 \put(100,32){$L_1$}
  \put(35,0){$L_2$}
   \put(92,30){$0$}
   \put(80,30){$1$}
    \put(65,30){$2$}
   \put(58,30){$3$}
   \put(50,30){$4$}
   \put(34,30){$3$}
   \put(22,29){$2$}
   \put(12,28){$1$}
   \put(8,28){$0$}  
 \end{overpic}
\caption{ $L_1,L_2$ and the Maslov index of the intersections}
\label{fig-10}
\end{figure}
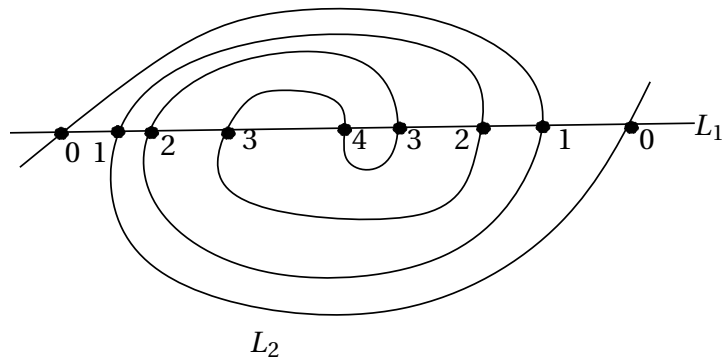

In the following lemma and proof, we will use the green and pink $A, A', B, B'$ represented in green and pink on Figure \ref{fig-11}. In order to describe them, note that the point $t_0=(\pi, 0)$ belongs to $L_2\cap L_1$ and let $t_-$ and $t_+$ be the two points of $L_2\cap L_1$ which are adjacent to $t_0$ and such that $t_-$ is on the left of $t_0$ and $t_+$ on its right. Note that $L_1$ has a canonical orientation, which induces an orientation of $L_2$. We let:  
  \begin{itemize}
  \item $A$ be the area enclosed by the oriented segments of $L_1$ and $L_2$ that go from $t_0$ to $t_-$.
  \item $A'$ be the area enclosed by the oriented segments of $L_1$ and $L_2$ that go from $t_0$ to $t_+$.
  \item $B$ be the area enclosed by the oriented segments of $L_1$ and $L_2$ that go from $t_+$ to $t_0$.
  \item $B'$ be the area enclosed by the oriented segments of $L_1$ and $L_2$ that go from $t_-$ to $t_0$.
  \end{itemize}

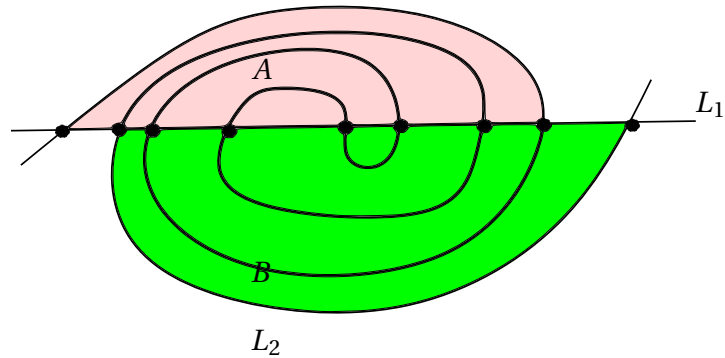
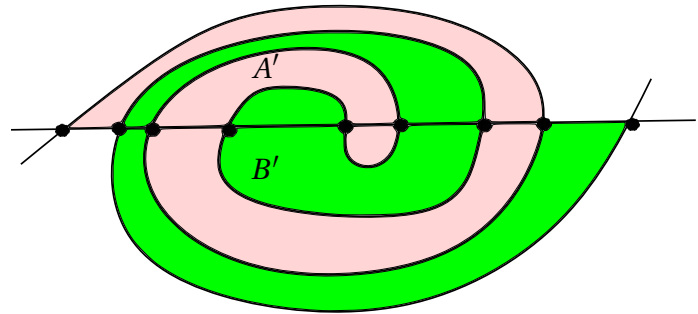
\begin{figure}
\begin{subfigure}{\textwidth}
\centering
 \begin{overpic}[width=9cm]{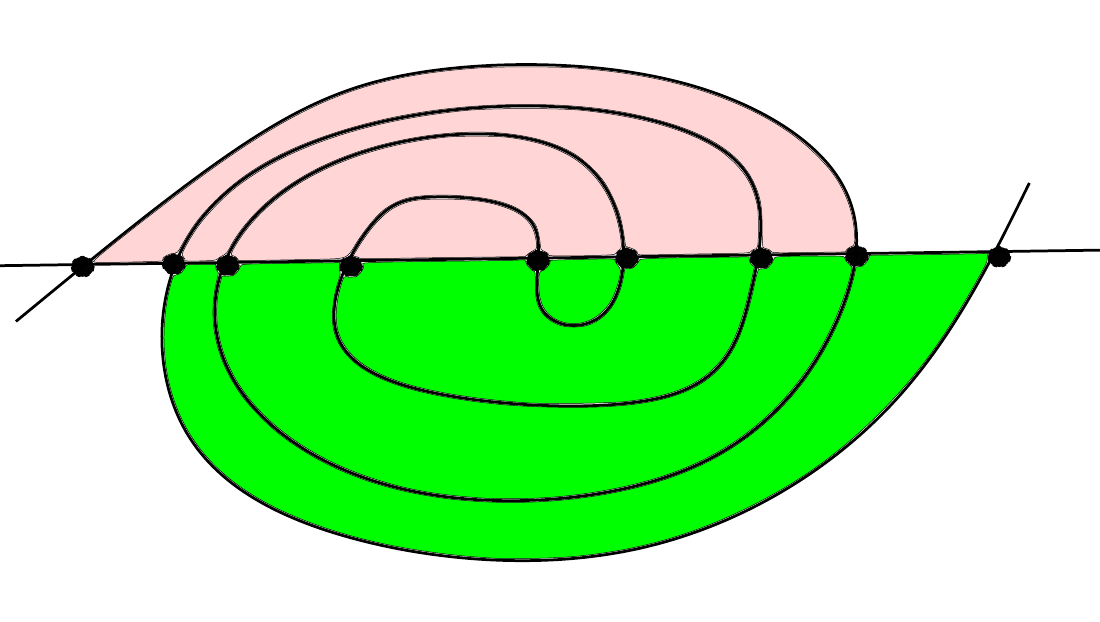}
 \put(100,35){$L_1$}
  \put(35,0){$L_2$}
   \put(35,10){$B$}
    \put(35,40){$A$}
 \end{overpic}
  \caption{The areas $A$ in pink and $B$ in green}\label{subfig-11a}
 \end{subfigure}
\begin{subfigure}{\textwidth}
\centering
\begin{overpic}[width=9cm]{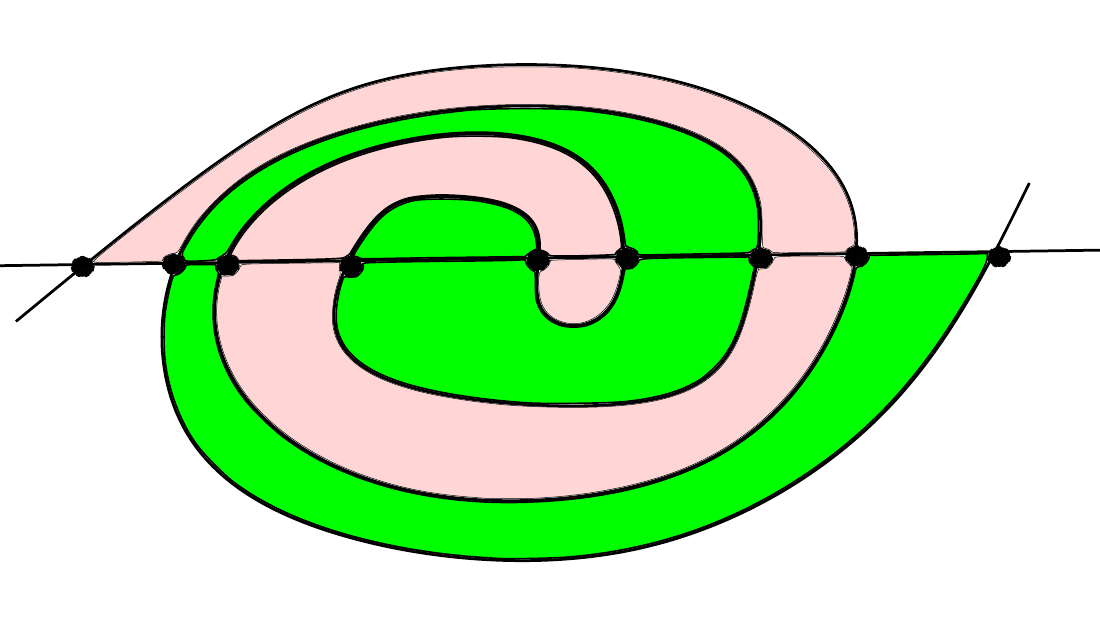}
 \put(35,25){$B'$}
    \put(35,40){$A'$}
 \end{overpic}
 \caption{ The  areas $A'$ in pink and $B'$ in green}\label{subfig-11b}
 \end{subfigure}
 
\caption{ $\gamma(L_1,L_2)$ is bounded below by the smallest of the  areas $A,A',B,B'$}
\label{fig-11}
\end{figure}

We are now ready to state 

\begin{lem} \label{Lemma-6.4}
  Assume $L_1$ is the zero section and $L_2$ is a smooth spiral as described above. Then, 
\[\gamma(L_1,L_2)\geq \min(A,A',B,B').\]
\end{lem}

\begin{figure}[ht]
 \begin{overpic}[width=9cm]{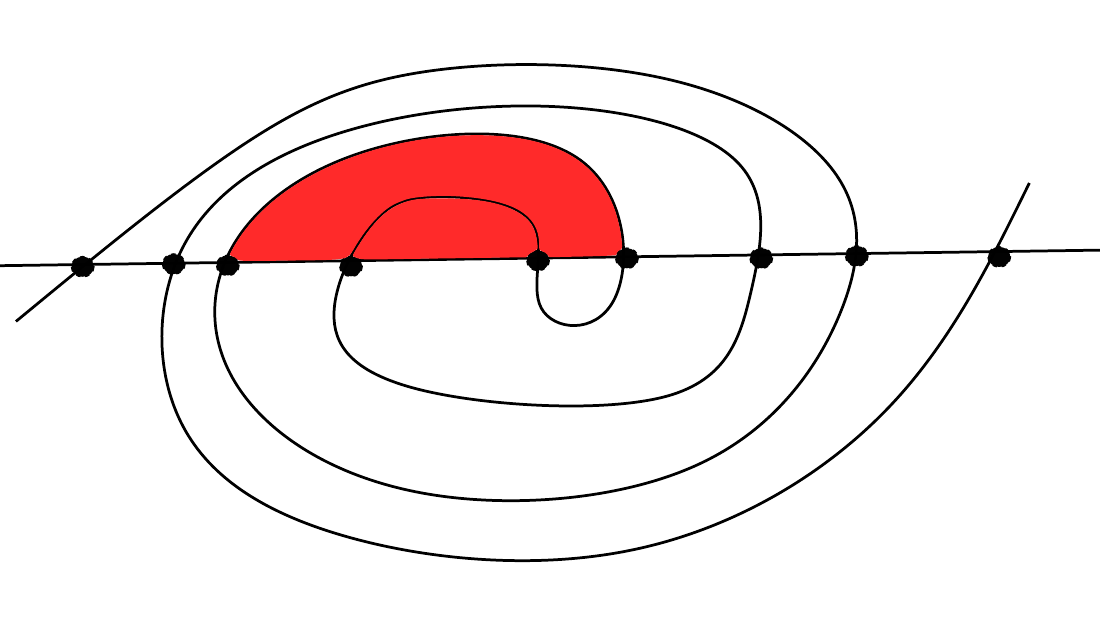}
 \put(100,32){$L_1$}
  \put(35,0){$L_2$}
   \put(92,30){$0$}
   \put(80,30){$1$}
    \put(65,30){$2$}
   \put(58,30){$3$}
   \put(50,30){$4$}
   \put(34,30){$3$}
   \put(22,29){$2$}
   \put(12,28){$1$}
   \put(8,28){$0$}
 \end{overpic}
  \begin{overpic}[width=9cm]{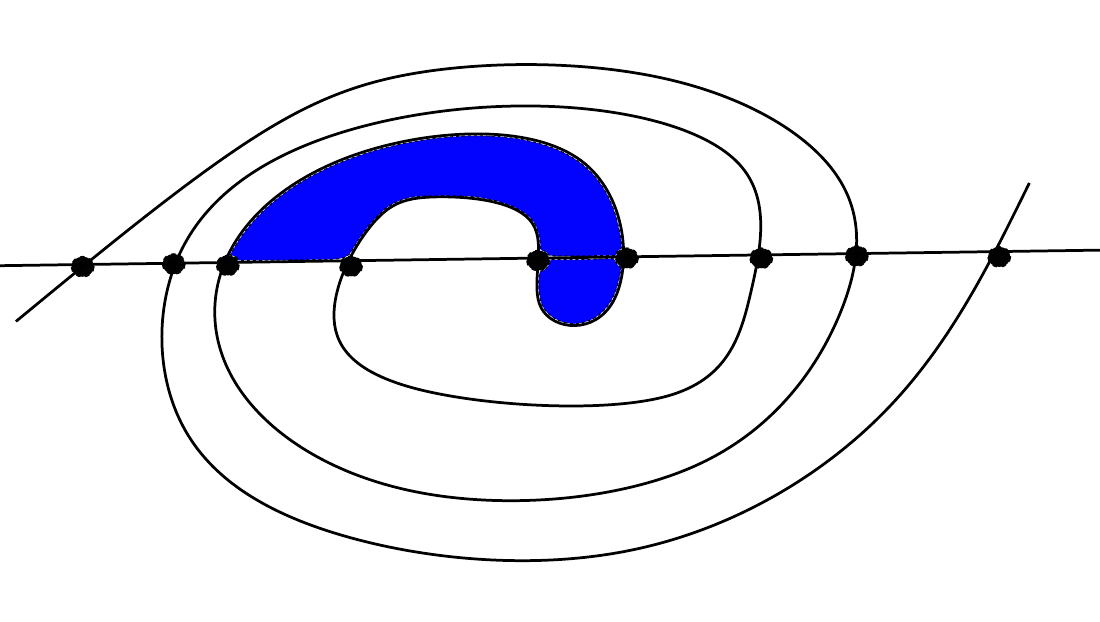}
 \put(100,32){$L_1$}
  \put(35,0){$L_2$}
   \put(92,30){$0$}
   \put(80,30){$1$}
    \put(65,30){$2$}
   \put(58,30){$3$}
   \put(50,30){$4$}
   \put(34,30){$3$}
   \put(22,29){$2$}
   \put(12,28){$1$}
   \put(8,28){$0$}
 \end{overpic}
 
\caption{ Areas realizing some of the possible values of $\gamma(L_1,L_2)$}
\label{fig-10b}
\end{figure}

\begin{figure}[ht]
\begin{overpic}[width=9cm]{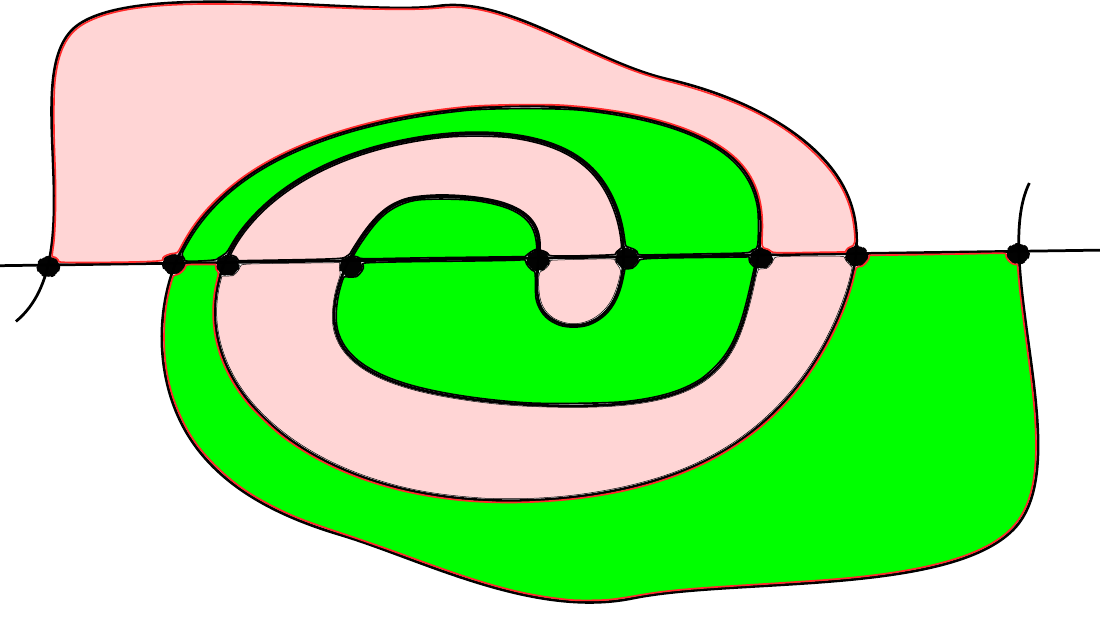}
 \put(35,25){$B'$}
    \put(35,40){$A'$}
 \end{overpic}
 \caption{ The  deformation $L_2(t)$}\label{fig-15}	
\end{figure}

\begin{figure}[ht]
 \begin{overpic}[width=9cm]{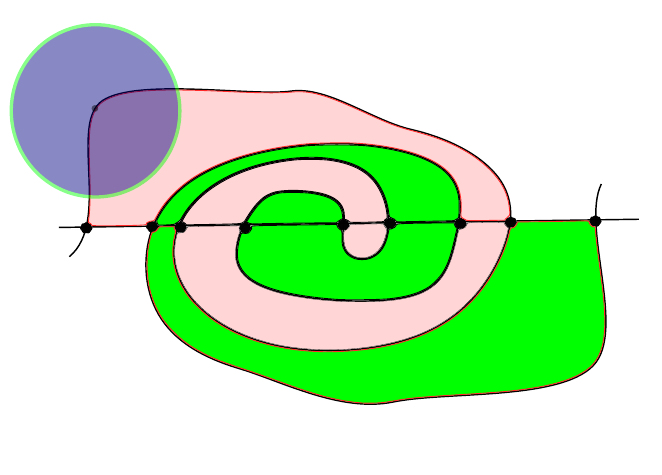}
 \end{overpic}
\caption{The disc centered on the Lagrangian}
\label{fig-8}
\end{figure}

\begin{figure}[ht]
\begin{overpic}[width=5cm]{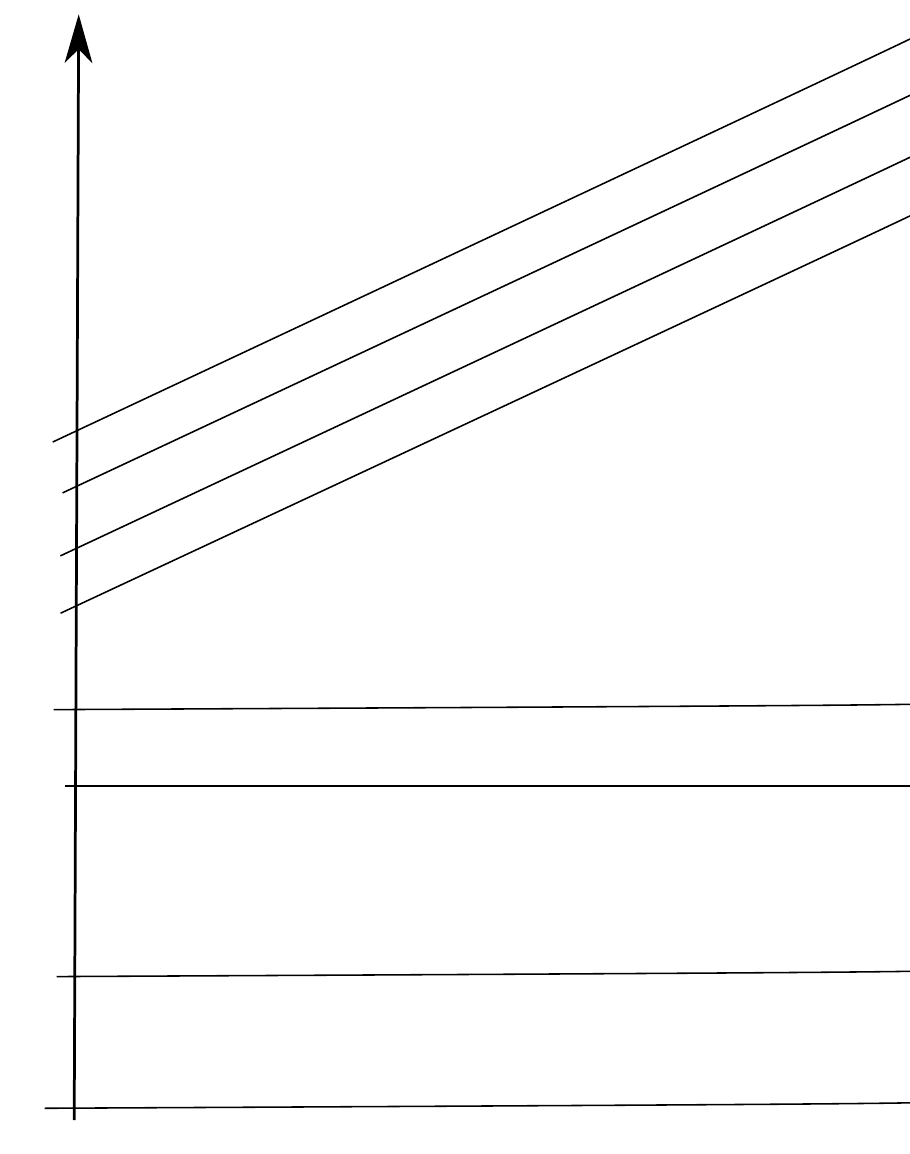}
\put(-8,3){$F(t)$}
\put(-8,12){$E(t)$}
\put(-8,30){$D(t)$}
\put(-8,38){$C(t)$}
\put(-8,55){$B'(t)$}
\put(-8,45){$B(t)$}
\put(-8,50){$A'(t)$}
\put(-8,60){$A(t)$}
 \end{overpic}
\caption{Actions of the intersection points $L_1\cap L_2(t)$.}
\label{fig-14}
\end{figure} 

\begin{proof} 
 Assume $\gamma(L_1,L_2)$ is less than the minimum $m$ of these four numbers. 
 We point out that $\gamma(L_1,L_2)$ must equal the difference of action between two points of consecutive index. Some of the possible values obtained are represented on Figure \ref{fig-10b} as the area of the red or blue discs.
 
 We see that if the area of the disc is less than $m$, then the boundary of the disc does not touch either of the outside curves bounding the green or pink domain on Figure \ref{fig-11}. Indeed the only discs connecting two consecutive points and having one of the two outside curves as a boundary are the pink or green discs in Figure \ref{fig-11}, subfigure \ref{subfig-11a} or \ref{subfig-11b}. 
 
    We may then continuously deform $L_2$ as $L_2(t)$ (see Figure \ref{fig-15}) so that we increase the area of $A(t), A'(t), B(t), B'(t)$ by inflating the lower boundary of the green part and (to preserve exactness) simultaneously increase the upper boundary of the pink part. 
    Then only $A(t),B(t),A'(t), B'(t)$ change, and so the values of the areas of the disc representing $\gamma(L_1,L_2(t))$ are unchanged by assumption, and since this varies continuously and starts below $m$ this quantity remains constant. 
     But this is impossible, since for $t$ large enough we can find the image of a symplectic disc of arbitrarily large area such that the image of a diameter goes to $L_1$ while the image of the disc avoids $L_2(s)$ and this implies $\gamma(L_1,L_2(s))$ is larger than half the area of such a disc (see \cite[Lemma 7]{Hum-Lec-Sey2}). 
         This contradicts our assumption. 
   \end{proof}

    By a simple approximation argument, we can show
\begin{lem}  \label{Lemma-6.4-for-infinite-spirals} Lemma \ref{Lemma-6.4} still holds for $L_2$ an infinite spiral.
\end{lem}

\begin{proof} Let $L_2$ be an infinite spiral. Recall that if $L\in\fL$ and if $\phi$ is a Hamiltonian diffeomorphism supported in a ball of radius $r$, then $\gamma(L, \phi(L))\leq \pi r^2$. Using this fact, we may construct a $\gamma$-Cauchy sequence representing $L_2$ by considering smooth spirals which coincide with $L_2$ outside balls centered at $(0,0)$ and whose radius goes to $0$. Lemma \ref{Lemma-6.4-for-infinite-spirals} then follows immediately by applying Lemma \ref{Lemma-6.4} to each element of the sequence.  
\end{proof}

We may now turn to the proof of Proposition \ref{Prop-6.9}.

\begin{proof}[Proof of Proposition \ref{Prop-6.9}] 
Consider the red  curve on Figure \ref{fig-3}. This represents the set $L_\alpha$ for some value of the parameter $\alpha$. Below, on Figure \ref{fig-12} we represented $L_\alpha$ and $L_\beta$. Obviously for  $0<\beta<\alpha$, the curve $L_\alpha$ tends to the origin faster than $L_\beta$. Now by an area preserving map, we can straighten $L_\beta$ to the zero section, and  then the pair $(L_\alpha, L_\beta)$ is equivalent to a pair of the type $(L_2,L_1)$, where $L_2$ is an infinite spiral and $L_1$ is the zero section. By Lemma \ref{Lemma-6.4-for-infinite-spirals}, $\gamma(L_\alpha, L_\beta)$ must be greater than the area between the curves, that is the area in light blue in Figure \ref{fig-13}. Notice that it is clear from Figure \ref{fig-11} that  $A+B=A'+B'$  and since our figure is now  symmetric with respect to the origin, we have $A=B, A'=B$ hence $A=A'=B=B'$. 
 
 \begin{figure}
 \begin{overpic}[width=12cm]{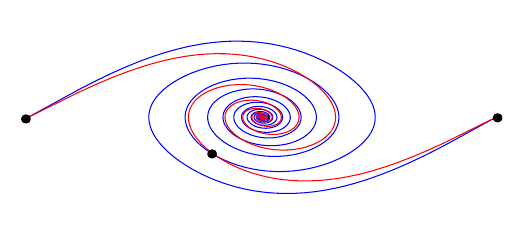}
 \put(80,20){\myRed{$\gamma_{\alpha,R}$}}
  \put(35,10){\myBlue{$\gamma_{\beta,R}$}}
  \put(20,27){\myRed{$\gamma_{\alpha,L}$}}
  \put(35,38){\myBlue{$\gamma_{\beta,L}$}}
 \end{overpic}
\caption{The first intersection points of  $L_\alpha$ (red) and $L_\beta$ (blue).}
\label{fig-12}
\end{figure}

\begin{figure}
 \begin{overpic}[width=12cm]{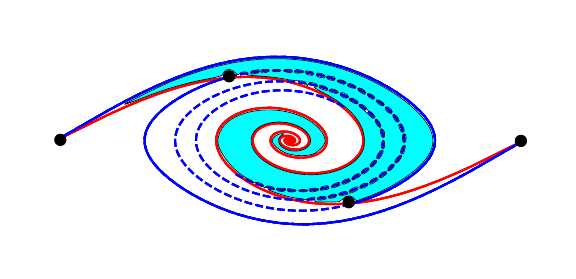}
 \put(80,22){\myRed{$L_\alpha$}}
  \put(45,40){\myBlue{$L_\beta$}}
 \end{overpic}
\caption{The  blue area is a lower bound for $\gamma(L_\alpha,L_\beta)$}
\label{fig-13}
\end{figure}

To estimate this area, we shall use energy estimates as follows.

Let $E_\alpha(t)$ be the energy of the pendulum at time $t$, i.e. $$E_\alpha(t)=E(\theta_\alpha(t),p_\alpha(t))=\frac{\dot\theta_\alpha(t)^2}{2} -\cos(\theta_\alpha(t))$$ 
Note that $E(\theta, p)= \frac{1}{2}p^2-\cos(\theta)  $ does not depend on $\alpha$.   By (\ref{eq:pendulum}), we have 
\begin{equation*}  \frac{d}{dt}\left ( \frac{\dot\theta_\alpha(t)^2}{2} -\cos(\theta_\alpha(t))\right)=-\alpha \dot\theta_\alpha(t)^2=-\alpha p_\alpha(t)\dot\theta_\alpha(t). \end{equation*} 

As a result we have that the area below $L_\alpha$ during the time interval $[t_0,t_1]$ is
$$\int_{t_0}^{t_1} p_\alpha \dot \theta_\alpha(t) dt= \frac{1}{\alpha}(E_\alpha(t_0)-E_\alpha(t_1)).
$$
    The curve $\gamma_{\alpha,L}$ on the left starts from  $E_\alpha(-\infty)=E(-\pi,0)=1$ (and the same holds for $\gamma_{\alpha,R}$). 
    Note that given $ \varepsilon >0$,  if we fix $\alpha< \varepsilon $ we can find $\beta <\alpha $ such that the first intersection point between 
    $\gamma_{\beta,L}$ and $\gamma_{\alpha,R}$ is as close as one wishes from the point $(-\pi,0)$. In other words $\gamma_{\beta,L}(t_\beta)=\gamma_{\alpha,R}(t_\alpha)$ and 
    $\gamma_{\alpha,R}\big((-\infty, t_\alpha)\big)$ is contained in a neighborhood of $(\pi,0)$.   
    
    Then we see from Figure \ref{fig-13} that $\gamma(L_\alpha, L_\beta)$ is bounded from below by the blue area surrounded by 
    $\gamma_{\beta,L}\big((-\infty, t_\beta]\big)\cup \gamma_{\alpha,R}\big([t_\alpha, +\infty)\big)$. The area surrounded by this curve will be
    $$\frac{1}{\beta}[1-E_\beta(t_\beta)]-\frac{1}{\alpha}[1-E_\alpha(t_\alpha)]$$
    But $$E_\alpha(t_\alpha)=E(\theta_\alpha(t_\alpha),p_\alpha(t_\alpha))=E(\theta_\beta(t_\beta),p_\beta(t_\beta))=E_\beta(t_\beta)$$ and 
    thus, setting $C(\alpha,\beta)=1-E_\alpha(t_\alpha)=1-E_\beta(t_\beta)$,
    the area we are trying to estimate is $C(\alpha, \beta) (\frac{1}{\beta}- \frac{1}{\alpha})$.  
    
    Notice that for fixed $\alpha$,  as $\beta $ goes to $0$, the intersection point converges to $(-\pi,0)$ and $$1-E_\beta(t_\beta)=\beta\int_{-\infty}^{t_\beta}p_\beta(t)\dot \theta_\beta(t)dt\simeq 8\beta$$
    since $\lim_{\beta\to 0}\int_{-\infty}^{t_\beta}p_\beta(t)\dot \theta_\beta(t)$ is the area under the separatrix of the frictionless pendulum, that is $\int_{-\pi}^\pi \sqrt{2(1+\cos(\theta)}d\theta =8$
    so the term $\frac{1}{\beta}[1-E_\beta(t_\beta)]$ is approximately equal to $8$, while the other term is $\simeq \frac{8\beta}{\alpha}$, so we get 
    
    $$\gamma(L_\alpha, L_\beta)\simeq 8\left(1- \frac{\beta}{\alpha}\right).$$
     
    However small is $\alpha$, we can choose $\beta$ to be smaller and such that the above quantity is greater than $\frac14(1-\frac\beta\alpha)$. 
\end{proof} 

\begin{rem} 
The curves $L_\alpha$ can be smoothed, while preserving the conclusion of Proposition \ref{Prop-6.9}. since the $L_\alpha$ are contained in a bounded subset of $T^*\bS^1$ and by Shelukhin's theorem (\cite{Shelukhin-Zoll}), the set of such Lagrangians is a $\gamma$-bounded set. However the above just proves that it is not compact !
\end{rem}

\section{Topological properties}\label{sec:topological-properties}

The goal of this section is to prove Theorem \ref{th:cohomology}. The proof will use spectral invariants for Hamiltonian diffeomorphisms and an ingredient from Lusternik-Schnirelman theory. We introduce the relevant material in \S~\ref{sec:Hamiltonian-spectral-invariants} (this will also be used in Section \ref{sec:weak-NLC}). The proof of Theorem \ref{th:cohomology} is then done in \S~\ref{sec:proof-cohomology}.

In this section we work on a cotangent bundle $M=T^*N$, endowed with the standard Liouville form $\lambda$. In particular, there is only one class of exact Lagrangians $\fL$ and branes $\LL$. Also recall that for any $L, L'\in\fL$, we have a canonical (up to shift of the grading) isomorphism $HF(L,L')\simeq H^*(N)$.

 \subsection{Hamiltonian spectral invariants}\label{sec:Hamiltonian-spectral-invariants}

 Spectral invariants and the spectral distance may also be defined for Hamiltonian diffeomorphisms similarly as in Section \ref{sec:preliminaries}  by using Hamiltonian Floer theory instead of Lagrangian Floer theory. In our setting this was first defined in \cite{Frauenfelder-Schlenk} (see also \cite{Lanzat}). The upshot is a collection of real numbers $c(\beta, \phi)$, for all non-zero cohomology classes $\beta\in H^*(N)$ and all compactly supported Hamiltonian diffeomorphisms $\phi$. These invariants are related to the Lagrangian spectral invariants by an inequality (see \cite[Prop 2.14]{M-V-Z}). Since we need a cohomological version and slightly different formulation from the original we state it as a
 \begin{lem}[\cite{M-V-Z}]
 We have for any $\beta \in H^*(N)$, for any exact Lagrangian brane $\tilde L$ and any $\phi\in\DHam_c(T^*N)$, the inequality
\begin{equation}\label{eq:Lagrangian-Hamiltonian-spectral-invariants}
  c(\beta, \phi)\leq \ell(\beta;\phi(\tilde L), \tilde L).
 \end{equation}
 \end{lem} 
  
\begin{proof} 
According to \cite[thm. 1.5]{Albers-PSS} and \cite[proof of Prop. 2.9]{M-V-Z} (adapting from the closed setting to the case of the cotangent bundle) we have the following diagram (open-closed map) that we translate from homology to cohomology
$$\xymatrix{HF_{a}^*(\phi)\ar[r]& HF_{a}^*(\phi(\widetilde L),\widetilde L)\\
H^*(M)\ar[u]^{j^*_a} \ar[r]^{i^*} & H^*(L)\ar[u]^{i^*_a}
}
$$
As a result if $a< c(\beta;\phi)$, the image of $\beta$ by $j_a^*$ vanishes and therefore the image of $i^*(\beta)$
 vanishes  in $HF_{a}^*(\phi(\widetilde{L}),\widetilde{L})$, hence $a < \ell (i^*(\beta);\phi(L),L)$. According to \cite{Fukaya-Seidel-Smith, Kragh2},  $i^*$ is an isomorphism in cohomology, so we write $\beta$ instead of  $i^*\beta$  by abuse of language. 
\end{proof}

 If $H$ is sufficiently $C^2$-small and autonomous, so that all the fixed points of $\phi_H^1$  correspond to constant orbits, then we have
 \begin{equation}
   \label{eq:spectral-morse}
   c(\beta, \phi_H^1)=\rho(\beta, H)
\end{equation}
where $\rho(\beta, H)$ denotes the Morse theoretic spectral invariant of $H$, defined by the following procedure.
Let $\nu:[0, +\infty)\to\R$ be a non-decreasing function  vanishing on some interval $[0,R]$ and  linear near infinity with small positive derivative on $(R,+\infty)$. 
We assume that $R$ is large enough so that the support of $H$ is included in $D_R^*N$, the cotangent disc bundle of radius $R$ (with respect to some given Riemannian metric). We set
\begin{equation}
  \label{eq:H_nu}
  H_\nu(x,p):=H(x,p)+\nu(\|p\|).
\end{equation}
 Since $H_\nu$ is proper, the Morse theoretic spectral invariants of $H_\nu$ may be defined by
 \[
 \rho(\beta;H_\nu)=\inf\left\{\,a\in\R\,:\ \beta \neq 0 \; \text{in}\; H_\nu^{<a} \right \} 
 \]
 where $H_\nu^{<a}=\{z \in T^*N \mid H_\nu(z)<a\}$. 
Observe that $\rho(\beta;H_\nu)$ is a critical value of $H$. A deformation argument then shows that the value $\rho(\beta;H_\nu)$ does not depend on the choice of the function $\nu$. Indeed, $H_\nu$ has no critical point in the complement of $D_R^*N$ and if $\nu_1,\nu_2$ are two such functions, the same holds for $H_s=(1-s)H_{\nu_1}+s H_{\nu_2}$ for $s\in[0,1]$. As a result $\rho(\beta,H_s)$ takes values in a fixed set of critical values, hence is constant by Sard's Lemma.  
We thus define $\rho(\beta, H)$ as $\rho(\beta, H_\nu)$ for any choice of $\nu$.

Since we will use them in Section \ref{sec:weak-NLC}, we introduce here the spectral distance on the group $\DHam_c(T^*N)$ and its completion. 
We will use the following notation:
 \[c_-(\phi)=c(1, \phi),\quad c_+(\phi)=-c(1, \phi^{-1}) \quad\text{and}\quad \gamma(\phi)=c_+(\phi)-c_-(\phi).\]
 The spectral distance $\gamma$ on $\DHam_c(T^*N)$ is then defined by 
 \[\gamma(\phi, \psi)=\gamma(\psi^{-1}\phi).\]
We denote by $\hatDHam(T^*N)$ the completion of $\DHam_c(T^*N)$ with respect to $\gamma$.
 
It can be checked using inequality (\ref{eq:Lagrangian-Hamiltonian-spectral-invariants}) that the Lagrangian $\gamma$-distance introduced in Section \ref{sec:preliminaries} interacts with the above \emph{Hamiltonian} $\gamma$-distance  by the following inequality:
\begin{equation*}\label{eq:ham-lag-gamma}
\gamma(\phi(L), \psi(L'))\leq \gamma(L,L')+\gamma(\phi, \psi),  
\end{equation*}
 for any $L, L'\in\fL$ and any $\phi, \psi\in\DHam_c(T^*N)$. As a consequence, the map $\DHam_c(T^*N)\times \fL\to\fL$, $(\phi,L)\mapsto\phi(L)$ extends continuously to a map $\hatDHam_c(T^*N)\times \hatL\to\hatL$. We still denote by $\phi(L)$ the result of this extended map for $\phi$ and $L$ in the completions.

\subsection{Proof of  Theorem \ref{th:cohomology}}\label{sec:proof-cohomology}
 
Our proof will use the following lemma from the classical Lusternik-Schnirelman theory.

\begin{lem}[Lusternik-Schnirelman \cite{Lusternik-Schnirelman}(see also \cite{Viterbo-Montreal})]\label{lem:Lusternik-Schnirelman} Let $f:M\to\R$ be a proper\footnote{The properness condition can be relaxed to the so-called Palais-Smale condition but we will not need it here.} function on a manifold $M$. Let $a\in H^*(M)$  and $b\in H^*(M)$ be classes with $\deg(b)>0$ and $a\cup b\neq 0$. If $\rho(a\cup b,f)=\rho(a, f)$, then $b$ induces a non-zero class in any neighborhood of the critical locus
  \[\{x\in M: df(x)=0, f(x)=\rho\}\]
  where $\rho$ denotes the common value  $\rho(a\cup b,f)=\rho(a, f)$.
\end{lem}

The above lemma in particular applies in the situation where $M=T^*N$, $f=H_\nu$ $a=\pi^*\alpha$ and $b=\pi^*\beta$ with the notations of Section \ref{sec:Hamiltonian-spectral-invariants}. 
We are now ready for the proof of Theorem \ref{th:cohomology}.

\begin{proof} To prove Theorem \ref{th:cohomology}, we need to prove that for any open neighborhood $U$ of $\gammasupp(L)$, the natural map $H^\ell(N)\to H^\ell(U)$ is injective for any integer $\ell$.  We may assume without loss of generality that the closure of $U$ is compact. Since this map sends the class $1\in H^0(N)$ to $1\in H^0(U)$, we may assume that $\ell$ is positive. 
Let $U$ be a compact open neighborhood of $\gammasupp(L)$ and let $\beta\in H^\ell(N)$ be a non-zero cohomology class of positive degree $\ell$.

Let $V$ be a closed neighborhood of $\gammasupp(L)$ included in $U$ and $h$ be a Hamiltonian diffeomorphism generated by some compactly supported $C^2$-small autonomous Hamiltonian $H$ which is equal to some constant $-\varepsilon$ on $V$ and satisfies $H>-\varepsilon$ in the complement of $V$.
 
Let $\tilde L$ be a lift of $L$ (this exists by the first item in Remark
 \ref{rem:completions}).
  By (\ref{eq:Lagrangian-Hamiltonian-spectral-invariants}), we have 
  \[c(\beta, h)\leq \ell(\beta; h(\tilde L),\tilde L).\]
  Since $H$ is $C^2$ small, the above left hand side coincides with the Morse theoretic spectral invariant $\rho(\beta, H)$. On the other hand, since $H=-\varepsilon$ on a neighborhood of $\gammasupp (L)$, Lemma \ref{lem:shift-of brane} implies that $h(\tilde L)=T_{-\varepsilon}\tilde L$. Thus, by the shift property of Lagrangian spectral invariants, the right hand side satisfies:
  \[\ell(\beta; h(\tilde L),\tilde L)=\ell(\beta; \tilde L,\tilde L)-\varepsilon=-\varepsilon=\min(H).\]
As a conclusion, we obtain
\[\rho(\beta, H)\leq\min(H).\]

By construction, we have $\rho(\beta, H)=\rho(\beta, H_\nu)$ where $H_\nu$ is given by (\ref{eq:H_nu}). Since $\min(H)=\rho(1,H)$, we obtain $\rho(\beta, H_\nu)=\rho(1, H_\nu)$. We may now apply Lusternik-Schnirelman theory (Lemma \ref{lem:Lusternik-Schnirelman}) to $H_\nu$. We deduce that $\beta$ induces a non-zero class in any neighborhood of the min locus of $H$, which is nothing but $V$. In particular, $\beta$ induces a non-zero class in $U$. This concludes the proof that the map $H^\ell(N)\to H^\ell(U)$ is injective.
\end{proof}

\section{The weak nearby Lagrangian conjecture}\label{sec:weak-NLC}

We let $(M, \omega)$ with $\omega=-d\lambda$ be a Liouville manifold obtained by completing a Liouville domain. Recall that this implies in particular that the vector field $X_{\lambda}$ defined by $i_{X_{ \lambda}}\omega=\lambda$ is complete. 
Let   $\fL\in I(M,\omega)$ be a non-empty class of Lagrangian submanifolds, so that the spectral distance  is well defined for $L, L'$ in $\fL$.  

\begin{lem} 
Let $\rho$ be a CS diffeomorphism of $(M, \omega)$ with conformal factor $a$. 
Then the map 
\begin{gather*} \Theta : \widehat{\DHam}(M, \omega) \longrightarrow  \widehat{\DHam}(M, \omega) \\
u \longmapsto  \rho u \rho^{-1} \end{gather*} 
is Lipschitz of ratio $a$.
\end{lem} 
\begin{proof} 
We have $\gamma ( \rho u \rho^{-1} ,  \rho v \rho^{-1})= \gamma ( \rho uv^{-1} \rho^{-1})=a \gamma ( uv^{-1})$  by the basic properties of $\gamma$.
\end{proof} 
 \begin{prop}[Weak conjugacy of conformally symplectic maps]\label{prop:weak-conjugacy}
 Let $\rho, \sigma$ be two conformally symplectic maps such that $\rho^{-1} \sigma$ is in ${\DHam}_c(M, \omega)$ (in particular they have the same ratio and coincide in the complement of a compact set). Then there exists $u \in \widehat{\DHam}(M, d\lambda)$ such that
 \begin{equation}
\sigma = u^{-1}\rho u.\label{eq:weak-conjugacy}
\end{equation}
\end{prop}

 As written above, equality (\ref{eq:weak-conjugacy}) abuses notation since it is not a priori defined what it means for a conformally symplectic map to be composed with an element of $\widehat{\DHam}(M,\omega)$. This equality should be understood as:
  \[(\sigma^{-1}u^{-1}\rho) u=\id\quad\text{in}\ \widehat\DHam(M,\omega). \]
  Indeed, $u \mapsto \sigma^{-1}u^{-1}\rho$ defines an isometry of $\DHam(M,\omega)$, so extends to $\widehat{\DHam}(M,\omega)$. 

\begin{proof} 
 Let $\psi \in \DHam (M, \omega)$ and consider the map $ u \mapsto \rho u \rho^{-1} \psi$. This is a contraction provided the ratio $a$ of $\rho$ is smaller than $1$ hence has a unique fixed point. Choose $\psi=\rho\sigma^{-1}$, so that this fixed point now satisfies $$\rho u \rho^{-1} \rho\sigma^{-1}=u$$ that is $u\sigma u^{-1}=\rho$, so $u$ realizes the conjugation. 
\end{proof} 

\begin{prop} \label{prop:exact-are-fixed-points}
Let $0<a<1$ and  $L_{0} \in \fL$ be an exact Lagrangian. Then there exists a Liouville form $\mu_0$ which differs from $\lambda$ by the differential of a compactly supported function, and whose associated Liouville flow  $\rho_0^t$  is complete and satisfies $\rho_0^t(L_{0})=L_{0}$ for all $t\in\R$.
\end{prop}

 \begin{proof} 
 Consider $U$ a Weinstein neighbourhood of $L_0$ with Liouville form $\lambda_{0}$. By exactness of $L_0$, there is a function $f$ defined on $U$ such that $\lambda=\lambda_{0}+ df$.
 Let $\beta:U\to\R$ be a smooth function equal to $1$ near $L_{0}$ and to $0$ near $\partial U$, and consider the form $\mu_0 = \lambda - d(\beta f)$. Obviously $d\lambda=d\mu_0$ and $\mu_0 = \lambda_{0}$ near $L_{0}$. As a result the vector field $X_{\mu_0}$ defined by $\mu_0=i_{X_{\mu_0}}\omega$ coincides with $X_{\lambda}$ outside of $U$ and with $X_{\lambda_{0}}$ in a neighbourhood of $L_0$. Thus the Liouville form $\mu_0$ suits our needs. Since it differs from $\lambda$ only on a compact set, its flow is also complete.
 \end{proof} 

 We are now ready to prove Theorem \ref{th:weak-NLC}, whose statement we recall.
 \setcounter{section}{1}
 \setcounter{thm}{7}
 \begin{thm}[Weak nearby Lagrangian conjecture]\label{Theorem-A4}
 Let $L_{0}, L_{1} \in \fL$. Then there exists $\phi \in \widehat{\DHam}(M, \omega)$ such that $\phi(L_{0})=L_{1}$.
 \end{thm} 
 \setcounter{section}{7}
 \setcounter{thm}{3}
 \begin{proof} 
 According to Proposition \ref{prop:exact-are-fixed-points}, there exist Liouville forms $\mu_0, \mu_1$ whose associated Liouville flows $\rho_{0}^t, \rho_{1}^t$ are complete and satisfy $\rho_{j}^t(L_{j})=L_{j}$ for $j=0,1$. Moreover, $\mu_{1}=\mu_{0}+ dg$ for some compactly supported function $g$.
 As a result,
\begin{gather*}   \frac{d}{dt} \rho_{0}^{t}\rho_{1}^{-t}(x) =X_{\mu_{0}}(\rho_{0}^{t}\rho_{1}^{-t}(x))-d\rho_{0}^{t}(X_{\mu_{1}}(\rho_{1}^{-t}))=\\
X_{\mu_{0}}(\rho_{0}^{t}\rho_{1}^{-t}(x))-d\rho_{0}^{t}(X_{\mu_{1}}(\rho_{0}^{{-t}}\rho_{0}^{t}\rho_{1}^{-t})).
\end{gather*}
Thus, $\rho_{0}^{t}\rho_{1}^{-t}$ is the flow of $$Z_t(x)=X_{\mu_{0}}(x)-d\rho_{0}^{t}(\rho_{0}^{-t}(x))X_{\mu_{1}}(\rho_{0}^{-t}(x))$$
and $i_{Z}\omega=\mu_{0}-i_{(\rho_{0}^{t})_* X_{\mu_{1}}} \omega$. Since $i_{Y}f^{*}(\alpha)=f^{*}(i_{f_{*}Y}\alpha)$  and  $(\rho_j^{t})^{*}\omega= e^{t}\omega$ for $j=0,1$ we have
\begin{align*}
  i_{Z}\omega &= \mu_{0}-e^t(\rho_{0}^{-t})^*i_{X_{\mu_{1}}}\omega\\
  & =\mu_{0}-e^t(\rho_{0}^{-t})^*\mu_1\\
              & = \mu_0- e^t(\rho_0^{-t})^*\mu_0-e^t(\rho_0^{-t})^*dg\\
  & = d \left(e^t \cdot g\circ {\rho_0^{-t}}\right)
\end{align*}
so $Z$ is a Hamiltonian vector field.

Let  $\rho_{0},\rho_{1}$ denote $\rho_0^t, \rho_1^t$ for some negative $t$. Then, Proposition \ref{prop:weak-conjugacy} applies and we deduce that $\rho_{0}$ and $\rho_{1}$ are conjugate by an element  $u\in\widehat{\DHam}(M,\omega)$, that is
$u^{-1}\rho_{1}u=\rho_{0}$. The fixed points of $\rho_{0}$ and $\rho_{1}$ in $\hatL(M, \omega)$ are unique.
 But if $L_{0}$ is fixed by $\rho_{0}$ then $u(L_{0})$ is fixed by $\rho_{1}$. We may thus conclude that $u(L_{0})=L_{1}$.  
 \end{proof}

 In the case where $(M,\omega)$ is a cotangent bundle $T^*N$, as explained in Section \ref{sec:Lagrangian-spectral-inv}, there exists only one class of closed exact Lagrangian denoted $\fL (T^*N)$. We immediately deduce

 \begin{cor} 
Let $L, L'$ be two closed exact Lagrangians in $T^*N$. Then there exists $\psi \in \widehat \DHam (T^*N)$ such that $\psi(L)=L'$. 
\end{cor} 

In $T^*N$, we have a slightly more precise statement:
 \begin{cor}\label{Weak-Nearby-with-support}
 Let $L, L'\subset{\mathfrak L}(T^*N)$ contained in $DT^*N$. Then there exists a sequence $(\phi_k)_{k\geq 1}$ in $\DHam(T^*N)$ such that $\gamma-\lim(\phi_k(L))=L'$ and $\phi_k$ is  supported in $DT^*N$. 
 \end{cor}

\begin{proof}
Let us first prove that for $\psi^t(q,p)=(q,e^{-t}p)$, and $L\subset DT^*N$, we have that for $t\geq 0$, $\psi^t(L)$	is obtained by applying a Hamiltonian isotopy to $L$, and this isotopy can be assumed to be supported in $DT^*N$. Indeed, it is well known that exactness of $\psi^t(L)$ implies that the isotopy is realized by a Hamiltonian isotopy $\rho^t$. Since the image of $L$ remains in $DT^*N$ we may truncate the Hamiltonian outside $DT^*N$ (i.e. making it compact supported) without changing $\psi^t(L)=\rho^t(L)$. 
Now as $t$ goes to $+\infty$, $\psi^t(L)$ converges to $0_N$ since $\gamma(\psi^t(L),0_N)=e^{-t}\gamma(L,0_N)$. So by the preceding argument, we have two sequences $(\rho_k)_{k\geq 1}, (\rho'_k)_{k\geq 1}$ such that  $\rho_k(L)$ and $\rho'_k(L')$ $\gamma$-converge to $0_N$. Then fix $ \varepsilon >0$ and choose $k$ large enough so that $\gamma(\rho_k(L),0_N), \gamma( \rho'_k(L'),0_N)$ are both less than $ \varepsilon /2$. Now
$$\gamma((\rho'_k)^{-1}\rho_k(L),L')=\gamma(\rho_k(L),\rho'_k(L'))\leq \gamma(\rho_k(L),0_N)+ \gamma(\rho'_k(L'),0_N)=2 \varepsilon /2= \varepsilon 
$$
Since $\rho_k$ and $\rho'_k$ are supported in $DT^*N$ so is $\phi_k=(\rho'_k)^{-1}\rho_k$ and $\gamma-\lim((\rho'_k)^{-1}\rho_k(L))=L'$.
\end{proof}

Let us now give an example of application of the above. Let us remind the reader of the following conjecture
\begin{Conjecture}\label{Conj1} Let $g$ be a Riemannian metric on the closed manifold $N$. There exists a constant $C_N(g)$ such that for any Lagrangian  $L$ in ${\mathfrak L}(T^*N)$ contained in $D_1T^*N=\{(q,p)\mid \vert p \vert_g\leq 1\}$ we have $$\gamma(L)\leq C_N(g)$$
\end{Conjecture}
A variant of this conjecture is 
\begin{Conjecture}\label{Conj2} Let $g$ be Riemannian metric on the closed manifold $N$. There exists a constant $C_N(g)$ such that for any Lagrangian  $L=\phi_H(0_N)$ that is the image of the zero section by a Hamiltonian map and contained in $D_1T^*N$ satisfies $$\gamma(L)\leq C_N(g)$$
\end{Conjecture}
Obviously the first conjecture is stronger than the second one. But in fact we have
\begin{prop} 
Conjecture \ref{Conj1} and Conjecture \ref{Conj2} are equivalent. 
\end{prop}
\begin{proof} 
Obviously it is enough to prove that Conjecture \ref{Conj2} implies Conjecture \ref{Conj1}. So let $L\in {\mathfrak L}(T^*N)$ and assume it is contained in $DT^*N$. By Corollary \ref{Weak-Nearby-with-support}, there exists $\phi_k$ in $\DHam (T^*N)$ supported in $DT^*N$ such that $\phi_k(0_N)$  $\gamma$-converges to $L$. In other words
$\lim_k \gamma(\phi_k(0_N),L)=0$. But then $\phi_k(0_N)$ is contained in $DT^*N$ (since $\phi_k$ has support in $DT^*N$ and $0_N\subset DT^*N$ !) and Conjecture \ref{Conj2} claims that $\gamma(\phi_k(L))\leq C_N(g)$. But then $\gamma(L)=\lim_k \gamma(\phi_k(0_N))\leq C_N(g)$ which proves Conjecture \ref{Conj1}. 
\end{proof}  

Let us conclude by stating the following
\begin{metatheorem} 
	Let $(M,\omega)$ be a Liouville manifold. Let $L$ be an exact Lagrangian. Then any closed property involving only $\gamma$ which is true for all Lagrangians $L'$, Hamiltonianly isotopic to $L$,  holds for any $L'$  in the same Floer theoretic class as $L$. 
\end{metatheorem}
\begin{proof} 
Indeed let $\psi^t$ be the flow of the Liouville vector field. Then $\psi^t(L')$ is Hamiltonianly isotopic to $L'$ and $\psi^t(L)$ is Hamiltonianly isotopic to $L$. So $\psi^t(L')=\rho^t(L'), \psi^t(L)=\sigma^t(L)$. Since $\gamma (\psi^t(L'),\psi^t(L))$ goes to zero, we have $\gamma(\rho^t(L'),\sigma^t(L))= \gamma(L', (\rho^t)^{-1}\sigma^t(L))$ goes to zero as $t$ goes to infinity. So if 
$(\rho^t)^{-1}\sigma^t(L))$ belongs to some interval, the same holds for $\gamma(L')$. 	
\end{proof}

\begin{rem} 
 In the case of cotangent bundles, where we have a single class, a statement concerning the $\gamma$ metric will hold for exact Lagrangians if it holds for any Lagrangian Hamiltonianly isotopic to the zero section. 
\end{rem}

\appendix

\section{One lemma and two examples to three questions (by \textsc{Maxime Zavidovique})}\label{appendix}

\subsection{A lemma}
In order to understand the similarities and differences between viscosity and variational solutions of Hamilton-Jacobi equations, one needs to understand better viscosity solutions in general contexts. Under convex hypotheses as they admit a variational characterization, many dynamical properties help to visualize their behavior.  Without convexity, it seems to be less the case. We provide here a simple lemma that shows that in dimension one, such solutions may not behave too erratically.

Let us recall the definitions first. Let $N$ be any smooth manifold and $G : T^*N\times \R\ \to \R$ a continuous function. We consider the equation $G\big(x, D_x u ,u(x)\big )=0 $ where $u : N \to \R$ is a continuous function.

\begin{defn}\rm
If $u : N\to \R$ and $x \in N$ we define
\begin{itemize}
\item  the superdifferential of $u$ at $x$, denoted by $\partial^+u(x)$ as the set of $D_x\phi$ where $\phi : N\to \R$ is a $C^1$ function such that $u-\phi$ has a local maximum at $x$.
\item 
 the subdifferential of $u$ at $x$, denoted by $\partial^-u(x)$ as the set of $D_x\phi$ where $\phi : N\to \R$ is a $C^1$ function such that $u-\phi$ has a local minimum at $x$.

\end{itemize}
\end{defn}

We recall that if $D_x u$ exists then $\partial^+u(x)=\partial^-u(x) = \{D_xu\}$. Reciprocally if both  $\partial^+u(x)$ and $\partial^-u(x)$ are non empty, then $D_x u$ exists (see \cite{CanSin04}).

\begin{defn}\rm
We say that a continuous function $u :N\to \R$ is a viscosity solution of $G\big(x,D_x u , u(x)\big) = 0$ if for all $x\in N$,
\begin{itemize}
\item for all $p\in \partial^+u(x)$, then $G\big(x,p,u(x)\big) \leqslant 0$,
\item  for all $p\in \partial^-u(x)$,  then $G\big(x,p,u(x)\big) \geqslant 0$.
\end{itemize}
\end{defn}

We now focus on the case of a 1 dimensional manifold $N$. So in the following, $N$ is either the real line $\R$ or the unit circle $\mathbb{T}^1$. In particular, as our result is mainly local we can see $x_0$ as being in an open interval, and the order is then the usual order on $\R$.  Our result is the following:

\begin{lem}\label{main-lemma}
Let $u : N\to \R$ be a Lipschitz viscosity solution to $G\big(x,D_x u , u(x)\big) = 0$. Let $x_0\in N$ such that the set $\left\{p\in \R,\ \  G\big(x_0,p,u(x_0)\big) = 0\right\}$ has empty interior.  Then the (almost everywhere defined) function $u'$ admits a left limit $u'_-(x_0)$ at $x_0$ and a right limit $u'_+(x_0)$ at $x_0$. It follows that $u$ has a left derivative at $x_0$ that is $u'_-(x_0)$ and a right derivative at $x_0$ that is $u'_+(x_0)$. 
 \end{lem}
Note that the hypothesis that $u$ is Lipschitz is not very restrictive as, as soon as $G$ is coercive in $p$, then any solution $u$ is automatically Lipschitz.
As immediate corollaries we deduce
\begin{cor}\label{cormax}
Let $u : N\to \R$ be a Lipschitz viscosity solution to $G\big(x,u'(x) , u(x)\big) = 0$. Let $x_0\in N$ such that the set $\left\{p\in \R,\ \  G\big(x_0,p,u(x_0)\big) = 0\right\}$ has empty interior.  Then, 
\begin{itemize}
\item $G(x_0,u'_-(x_0),u(x_0)\big) = G(x_0,u'_+(x_0),u(x_0)\big)=0$.
\item if $ u'_-(x_0)\leqslant u'_+(x_0)$ then $\partial^-u(x_0) = [ u'_-(x_0) , u'_+(x_0)]$ and for all $p\in [ u'_-(x_0) , u'_+(x_0)]$ it holds  $G(x_0,p,u(x_0)\big)\geqslant 0$,
\item if $ u'_-(x_0)\geqslant u'_+(x_0)$ then $\partial^+u(x_0) = [ u'_+(x_0) , u'_-(x_0)]$ and for all $p\in [ u'_+(x_0) , u'_-(x_0)]$ it holds $G(x_0,p,u(x_0)\big)\leqslant 0$.
\end{itemize}

\end{cor}

\begin{proof}

If $x$ is a point where $u'(x)$ exists, then $G\big(x,u'(x),u(x)\big)=0$. By letting $x\to x_0$ from below we obtain that $ G(x_0,u'_-(x_0),u(x_0)\big)=0$ and similarly by letting $x\to x_0$ from above we obtain that $ G(x_0,u'_+(x_0),u(x_0)\big)=0$.

We then explain the next point. If $ u'_-(x_0)\geqslant u'_+(x_0)$ and $p\in [ u'_-(x_0) , u'_+(x_0)]$ let $v$ be defined in a neighborhood of $x_0$ by
$$
v(x) =
\begin{cases}
 u(x) + \big(p- u'_-(x_0)\big)(x-x_0) ,\quad {\rm if \ \ } x\leqslant x_0,\\
 u(x) + \big(p- u'_+(x_0)\big)(x-x_0) ,\quad {\rm if \ \ } x\geqslant x_0.
\end{cases}
 $$
 Then $v\leqslant u$, $v(x_0) = u(x_0)$ and $v'(x_0)$ exists with $v'(x_0) =p$. It follows easily that $\{p\} = \partial^- v(x_0) \subset \partial^- u(x_0)$. The other inclusion $\partial^-u(x_0) \subset [ u'_-(x_0) , u'_+(x_0)]$ is left as an exercise and the end is just the definition of viscosity solution.
\end{proof}

A key ingredient in the proof of lemma \ref{main-lemma} is the following elementary ``à la Darboux'' lemma\footnote{It was pointed out by A. Davini that this lemma also appears in \cite{Tran}.}:

\begin{lem}\label{dif}
Let $f : N \to \R$ be a continuous function. 
\begin{enumerate}
\item assume that there are $a<b$ and $p>p'$ such that $p\in \partial^- f(a)$ and $p'\in \partial^- f(b)$. Then for all $p>p_0>p'$ there exists $c\in (a,b)$ such that $p_0\in \partial^+ f(c)$;
\item assume that there are $a<b$ and $p<p'$ such that $p\in \partial^+ f(a)$ and $p'\in \partial^+ f(b)$. Then for all $p>p_0>p'$ there exists $c\in (a,b)$ such that $p_0\in \partial^- f(c)$.
\end{enumerate}

\end{lem}
\begin{proof}
Let us prove the first point. Up to adding a linear function to $f$, we assume $p_0=0$, thus $p>0>p'$. There exists a $C^1$ function $\phi$ such that $\phi'(a) = p$, $\phi\leqslant f$ with equality at $a$. This implies that $f(x)> f(a)$ if $x\in (a,a+\varepsilon)$ for some small $\varepsilon>0$. A similar argument yields that $f(x)> f(b)$ if $x\in (b-\varepsilon,b)$ for some $\varepsilon>0$. It follows that $f$ restricted to $[a,b]$ has a maximum that is reached at some $c\in(a,b)$. At this point we indeed have $0\in \partial^+ f(c)$.
\end{proof}

We finally turn to the 
\begin{proof}[Proof of Lemma \ref{main-lemma}]
Let us prove the existence of a left limit of the derivative. We argue by contradiction. As $u$ is Lipschitz if the result does not hold, we can find $p<p'$ and an increasing sequence $y_n \to x_0$ of derivability points of $u$ such that $u'(y_{2n}) \to p$ and $(y_{2n+1}) \to p'$. In particular, for $n$ large enough, we have $u'(y_{2n})<u'(y_{2n+1})$ and  $u'(y_{2n+2})<u'(y_{2n+1})$.

Let now $p'' \in(p,p')$ then for $n$ large enough, $p'' \in \big(u'(y_{2n}),u'(y_{2n+1})\big)$ and by our lemma, there is $z_n \in (y_{2n},y_{2n+1})$ such that $p'' \in \partial^+ u(z_n)$. In particular, by definition of viscosity solution $G\big(z_n,p'',u(z_n)\big)\geqslant 0$. Letting $n\to +\infty$ we find that $G(x_0,p'',u(x_0)\big) \geqslant 0$.

The same argument applied between the points  $y_{2n+1}$ and $y_{2n+2}$ gives $G(x_0,p'',u(x_0)\big) \leqslant 0$. Finally, we have proven that $G(x_0,\cdot ,u(x_0)\big)$ vanishes on $(p,p')$ which is a contradiction.

The existence of left and right derivatives follow from the property, for Lipschitz functions, that if $h\neq 0$, $h^{-1}\big(u(x_0+h)- u(x_0)\big) = h^{-1}\int_0^h u'(x_0+s)ds$.
\end{proof}

\subsection{Two examples and three questions}

\subsubsection
{\bf Question 1:} \emph{For a non-Tonelli Hamiltonian $H$, is it true that the viscosity solution is the graph selector of $\tilde L_\infty(H,\alpha)$?}

\medskip
This question arises as the result is proven to be true for Tonelli Hamiltonians in Corollary \ref{cor:KAMF-graph-selector}. We answer by the negative starting from the classical damped pendulum. Let $H: \mathbb T^1\times \R \to \R$ be defined by 
$$(x,p)\mapsto \frac12 p^2 - \cos(2\pi x)-1.$$

For $\alpha>0$ let us recall that there exists a unique (continuous) viscosity solution $u_\alpha : \mathbb T^1 \to \R$ to 
$$\alpha u_\alpha + H(x, D_x u_\alpha ) =0.$$

In the example of the damped pendulum (see Figure \ref{fig:pendulum+perturb}), the Birkhoff attractor $B_{H,\alpha}$  is given by the spiral in blue. The viscosity solution is given by integrating the top branch $f^+$ and the lower branch $f^-$ of the Birkhoff attractor (in red):
$$u_\alpha (x) = \begin{cases}
\int_{0}^x f^+(s)ds ,  \quad {\rm if}\  0\leqslant x\leqslant 1/2,\\
\int_{1}^x f^-(s)ds ,  \quad {\rm if} \ 1/2\leqslant x\leqslant 1.\\
\end{cases}
$$
In this instance, $u_\alpha$ is $C^1$ in $\mathbb T^1 \setminus \{1/2\}$ and at those points its derivative is the only sub-tangent and super-tangent. At $1/2$ the set of super-tangents is $[f^-(1/2), f^+(1/2)]$ and there are no sub-tangents.

Let us now set $H_1 = H+\rho$ where $\rho : \mathbb T^1\times \R \to \R$ is a smooth, (big) bump function such that (See Figure \ref{fig:pendulum+perturb})
 \begin{itemize}
 \item the support of $\rho$ does not intersect the Birkhoff attractor of $H$,
 \item the support of $\rho$ intersects  $\{1/2\} \times [f^-(1/2), f^+(1/2)]$ in such a way that $H_1$ has a huge maximum on this segment.
 \end{itemize}
 
 Note that $B_{H,\alpha} $ is still the Birkhoff attractor of $H_1$ and $\tilde L_\infty(H,\alpha)=\tilde L_\infty(H_1,\alpha)$ its associated brane. This follows from Lemma \ref{lem:shift-of brane}
 and the choice of the support of $\rho$.

\begin{figure}
\centering
\def\svgwidth{0.8\textwidth}
{\small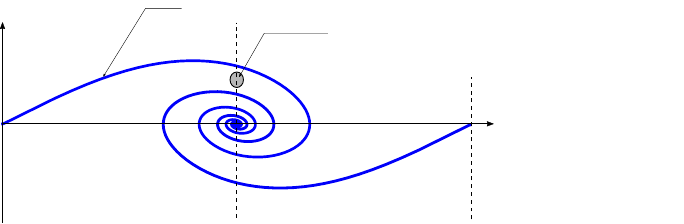}
\caption{The Birkhoff attractor of the pendulum $H$ and its perturbation $H_1=H+\rho$.}
\label{fig:pendulum+perturb}
\end{figure}

 It follows that the variational solution $\tilde u_\alpha$, for $H$ is also the variational solution for $H_1$ as it can be recovered only knowing $\widetilde L_\infty(H_1, \alpha)$. So $\tilde u_\alpha = u_\alpha$ as variational and viscosity solutions coincide for Tonelli Hamiltonians. 
 
 However $u_\alpha$ is no longer a viscosity solution for $H_1$. Indeed, it is false that $\alpha u_\alpha(1/2)+ H_1(1/2,y)\leqslant 0$ for all $y\in  [f^-(1/2), f^+(1/2)]$ thus violating the definition of viscosity solution.
 
 \subsubsection
{\bf Question 2:} \emph{For a non-Tonelli Hamiltonian $H$, is it true that, as in the Tonelli case, that if $u_\alpha$ is the viscosity solution of the $\alpha$-discounted Hamilton-Jacobi equation then $\overline{{\rm graph}(Du_\alpha)}\subset B_{H,\alpha}$? }

\medskip
For Tonelli Hamiltonians, this holds thanks to Theorem \ref{th:discounted}. Again the answer is negative and to prove it we use again the Hamiltonian $H_1$.

We now prove that 
\begin{prop}\label{A6}
If $\alpha >0$ and $u_\alpha : \mathbb T^1 \to \R$ is the viscosity solution to 
$$\alpha u_\alpha + H_1(x, D_x u_\alpha ) =0,$$
then 
$\overline{{\rm graph}(Du_\alpha)}\not\subset B_{H_1,\alpha}$ 
\end{prop}

\begin{proof}
Let us argue by contradiction. Recall that as $H_1$ is coercive, $u_\alpha$ is automatically Lipschitz hence derivable almost everywhere. For $x>0$ small, there is a unique point of $B_{H,\alpha}=B_{H_1, \alpha}$ above $x$, that we denoted above by $\big(x,f^+(x)\big)$. So $u'_\alpha(x) = f^+(x)$. Let $x_0$ be  the maximal point such that $u'_\alpha(x) = f^+(x)$ for $x\in (0,x_0)$. 

By Lemma \ref{main-lemma} and Corollary \ref{cormax}, $u_\alpha$ has a right derivative $u'_{\alpha+}(x_0) $ at $x_0$, by hypotheses, $\big(x_0,u'_{\alpha+}(x_0)\big) \in B_{H,\alpha} $ and $H_1\big(x_0,u'_{\alpha+}(x_0)\big) =-\alpha u_\alpha(x_0) = H_1\big(x_0,f^+(x_0)\big)$. 

Note that if $t\mapsto \big(x(t),p(t)\big)$ is a trajectory of $\phi_{-H, \alpha}$, then  $\frac{d}{dt} H\big(x(t),p(t)\big) = -\alpha p(t)^2$. It follows that $H$ is strictly decreasing on such non constant trajectories. As $H$ and $H_1$ coincide on $B_{H,\alpha}$ and that the latter is made of $2$ trajectories (and 2 fixed points), we deduce that there is only one other point $(x,y) \in B_{H,\alpha}$  such that $H_1(x,y) =  H_1\big(x_0,f^+(x_0)\big)$. By symmetry, this point is $\big(1-x_0, f^-(1-x_0)\big)$. Therefore we must have $x_0 = 1/2$ and as previously, we get a contradiction as $\alpha u_\alpha(1/2) +H_1(1/2,\cdot)$ takes positive values on the corresponding vertical segment.
\end{proof}

\subsubsection
{\bf Question 3:} \emph{Let $H : T^*N \to \R$ be a Tonelli Hamiltonian, $\alpha>0$ and $u_\alpha : N\to \R$ be the discounted weak KAM solution associated to the factor $\alpha$. The pseudograph of $u_\alpha $ is 
$$\mathcal G(u_\alpha) = \overline{ \big\{\big(x,u_\alpha'(x)\big), \ \ x\in \mathcal D\big\}}$$ where $\mathcal D$ is the set of differentiability points of $u_\alpha$. Is it true that $B_{H,\alpha} = \overline{\cup_{t\geqslant 0} \phi_{-H,\alpha}^t\mathcal G(u_\alpha) }$ ? }

\medskip
Note that this is the case for the damped pendulum.
We will construct an example on the annulus $\mathbb T^1\times \R$, for $\alpha>0$ fixed. Let $f^+ : [0,5/6] \to \R$ be smooth, as on Figure \ref{fig:question3} and $f^-(x) = -f^+(1-x)$.
Let $1/2<\varepsilon_1<\varepsilon_2< 3/4$ to be chosen later and let $v^+: [0,1]\to [0,1]$ be a smooth function such that $v^+(0) = 0$, $v^+ $ is increasing on $[0,1/4]$, $v^+$ is constant equal to $1$ on $[1/4, \varepsilon_1]$ and $v_+$ decreases to $0$ on $[\varepsilon_1,\varepsilon_2]$ to stay $0$ afterwards. Let us set $v^- (x) = -v^+(1-x)$.

\begin{figure}[h!]
\centering
\def\svgwidth{0.8\textwidth}
{\small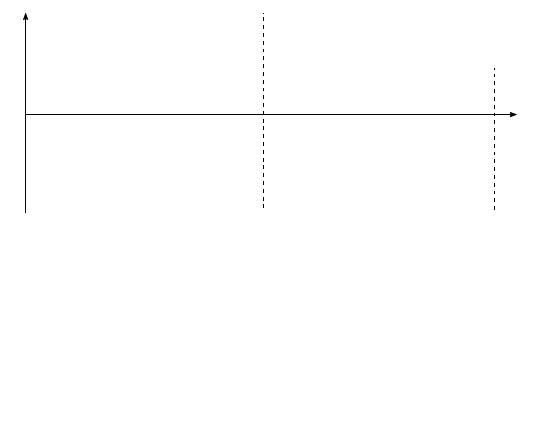}
\caption{The functions $f_\pm$ and $v_\pm$.}
\label{fig:question3}
\end{figure}

We want to construct $H$ such that the graph of $f^+$ restricted to $[0,\varepsilon_2]$
 is a trajectory of the flow $\phi_{-H,\alpha}$, and has  horizontal velocity $v^+$, and symmetrically for the graph of $f^-$ with velocity $v^-$. If $H(0,0)=0$ then the function $u_\alpha $ defined by 
 $$u_\alpha (x) = \begin{cases}
\int_{0}^x f^+(s)ds ,  \quad {\rm if}\  0\leqslant x\leqslant 1/2,\\
\int_{1}^x f^-(s)ds ,  \quad {\rm if} \ 1/2\leqslant x\leqslant 1,\\
\end{cases}
$$
will be the discounted solution and $\overline{\cup_{t\geqslant 0} \phi_{-H,\alpha}^t\mathcal G(u_\alpha) }\neq B_{H,\alpha}$ as it does not disconnect the annulus. Indeed as $(\varepsilon_2, f^+(\varepsilon_2))$ and $(1-\varepsilon_2, f^-(1-\varepsilon_2))$ are fixed points,  
$$\overline{\bigcup_{t\geqslant 0} \phi_{-H,\alpha}^t\mathcal G(u_\alpha) } \subset \{(x,f^+(x)) , \ \ x\in [0,\varepsilon_2]\} \cup \{(x,f^-(x)) , \ \ x\in [1-\varepsilon_2, 1]\}.$$ 
 
 When writing the conformal Hamiltonian flow we find the relations (for the relevant values of $x$)
 $$\begin{cases}
 \dot x(t) = \partial_p H\Big(x(t), f^\pm\big(x(t)\big)\Big) = v^\pm (x(t)) \\
\dot p(t) = -\partial_x H(x(t), f^\pm(x(t)))-\alpha f^{\pm}(x(t)) = v^\pm(x(t))(f^{\pm})'(x(t)).
 \end{cases}
$$
Integrating along a trajectory $x(t)$ that follows one of the graphs of $f^\pm$ we find that,
\begin{multline*}
H(x(T), f^\pm(x(T)))-H(x(0), f^\pm(x(0))) =\\
 \int_0^T \partial_x H\big(x(t), f^\pm(x(t))\big)  v^\pm (x(t)) +  \partial_p H\big(x(t), f^\pm(x(t))\big)(f^{\pm})'(x(t))  v^\pm (x(t)) dt \\
= -\int_0^T \alpha f^\pm(x(t))v^\pm (x(t)) dt . 
\end{multline*}
It follows that if we define $H(0,0) = H(1,0) = 0$ then 
\begin{equation}\label{lyap}
H(x,f^+(x))= \int_0^x \alpha f^+(s) ds \ \ \ {\rm and}\ \ \  H(x,f^-(x))= \int_1^x \alpha f^-(s) ds.
\end{equation}
In particular, we find that 
$$H(1/2, f^+(1/2)) =H(1/2, f^-(1/2)).$$ 
Therefore, as $\partial_p H(1/2, f^-(1/2))= v^-(1/2)=-1$ and $f^+(1/2)-f^-(1/2)=2$
$$H(1/2, f^+(1/2))-H(1/2, f^-(1/2))=0>-2= \partial_p H(1/2, f^-(1/2))(f^+(1/2)-f^-(1/2))$$
and similarly 
$$H(1/2, f^-(1/2))>H(1/2, f^+(1/2)) + \partial_p H(1/2, f^+(1/2))(f^-(1/2)-f^+(1/2)).$$
We then chose $\varepsilon_1<\varepsilon_2\in (1/2,3/4)$ in order to have 
\begin{equation}\label{cond1}
\forall x\in (1/2,\varepsilon_2], \quad H(x, f^+(x))>H(x, f^-(x)) + \partial_p H(x, f^-(x))(f^+(x)-f^-(x))
\end{equation}
which is possible by continuity and the fact that $ \partial_p H(x, f^-(x))(f^+(x)-f^-(x))= 2$ on the interval considered.

Since $\partial_p H(x, f^+(x))(f^-(x)-f^+(x)) = -2v^+(x)\leq 0$,  and using \big(see equation \eqref{lyap}\big), we conclude that if $x\in [1/2,\varepsilon_2]$, 
\begin{align}\label{cond2}
H(x,f^-(x))&> H(1/2,f^\pm(1/2))> 
H(x,f^+(x))\nonumber\\ &\geq H(x,f^+(x)) +\partial_p H(x, f^+(x))(f^-(x)-f^+(x)).
\end{align}

By symmetry of the construction with respect to the point $(1/2,1/2)$ we have inequalities similar to \eqref {cond1} and \eqref{cond2} for $x\in [1-\varepsilon_2,1/2]$ by switching $+$ and $-$. Those inequalities are necessary and sufficient to build our Tonelli Hamiltonian $H$. Here is a sketch of the construction.

Start by taking $\varepsilon'_2>\varepsilon_2$ in order that   \eqref {cond1} and \eqref{cond2} still hold for $x\in [1/2,\varepsilon'_2]$ and   $x\in [1-\varepsilon'_2,1/2]$. 

Define  $H_0$ on the strip 
\begin{multline*}
S_\varepsilon = \{(x,y), x\in [0,\varepsilon'_2], y\in [f^+(x)-\varepsilon, f^+(x)+\varepsilon]\}\\
\cup \{(x,y), x\in [1-\varepsilon'_2,1], y\in [f^-(x)-\varepsilon, f^-(x)+\varepsilon]\}
\end{multline*}
 for $\varepsilon>0$
 small enough, by
$H_0(x,y) = H(x,f^+(x)) + (y-f^+(x)) \partial_pH(x,f^+(x))-\varepsilon(y-f^+(x))^2$ if $(x,y)$ is in the first part of the strip, and  $H_0(x,y) = H(x,f^-(x)) + (y-f^-(x)) \partial_pH(x,f^-(x))- \varepsilon(y-f^-(x))^2$ otherwise. Then extend linearly $H_0$ on each $\{x\}\times [f^-(x)+\varepsilon, f^+(x)-\varepsilon]$. For $\varepsilon>0$ small enough the obtained function is convex in each fiber. 

Let $H_1 = H_0+\rho$ where $\rho $ is smooth, $C^2$ small, vanishes on the smaller strip $S_{\varepsilon/2} \subset S_\varepsilon$ and such that $y\mapsto \rho(x,y)$ is strictly convex for $y\in [f^-(x)+\varepsilon, f^+(x)-\varepsilon]$. Then the function $H_1$ is convex in each fiber (where it is defined).

For $x\in [1-\varepsilon'_2, \varepsilon'_2]$ define for $y>f^+(x)+\varepsilon$, 
$$H_1(x,y) = H_1(x,f^+(x)+\varepsilon) +  M(y-f^+(x)-\varepsilon)^2+ M(y-f^+(x)-\varepsilon)
$$
and for
$y<f^-(x)-\varepsilon$, 
$$H_1(x,y) = H_1(x,f^-(x)-\varepsilon) +  M(y-f^-(x)+\varepsilon)^2- M(y-f^-(x)+\varepsilon)
$$
where $M$ is a big enough constant to ensure strict convexity in the fibers.

Now we extend $H_1$ as follows for $x>\varepsilon'_2$:
\begin{itemize}
\item if $y< f^-(x)-\varepsilon$, 
$$H_1(x,y) = H_1(x,f^-(x)-\varepsilon) +  M(y-f^-(x)+\varepsilon)^2- M(y-f^-(x)+\varepsilon);
$$
\item  if $y> f^-(x)+\varepsilon$, 
\begin{multline*}
H_1(x,y)=(1- \psi(x))  H_1(\varepsilon'_2,f^-(\varepsilon'_2)+\varepsilon-f^-(x)+y) \\
+\psi(x) H_1(x,f^-(x)+\varepsilon) +  M(y-f^-(x)-\varepsilon)^2- M(y-f^-(x)-\varepsilon)
\end{multline*}
where $\psi : [\varepsilon'_2,1] \to [0,1]$ is a smooth, non--decreasing function that is $0$ in a neighborhood of $\varepsilon'_2$ and $1$ in a big neighborhood of $1$.
\end{itemize}
For a suitably chosen $\psi$, the Hamiltonian $H_1$ is strictly convex in the fibers.

We make a symmetric construction for $x<1-\varepsilon'_2$. Note that the obtained function is then $1$-periodic in $x$.

To finish, we just round off the corners of $H_1$, without modifying it in a neighborhood of the graphs of $f^+$ above $[0,\varepsilon_2]$ and $f^-$ on $[1-\varepsilon_2,1]$.

The Hamiltonian thus obtained has the required  trajectories as described at the beginning of the section.

\subsection*{Last acknowledgements}
M.Z. is very grateful to the authors for allowing him to include this appendix to their article. He also thanks M.-C. Arnaud for suggestions and comments on the first attempts to build those examples. Finally I express here my gratitude to V. Humili\`ere for a drastic simplification of the proof of Proposition \ref{A6} and for the figures of the Appendix.

\printbibliography

\end{document}